\documentclass[11pt,reqno]{amsart}
\usepackage {amssymb}
\usepackage {amsmath}
\usepackage {bbm}
\usepackage{amsthm}
\usepackage{mathtools}
\usepackage{graphicx}
\usepackage {amscd}
\usepackage {epic}
\usepackage{mathrsfs}
\usepackage{etaremune}

\usepackage[dvipsnames]{xcolor}
\usepackage[colorlinks,citecolor=OliveGreen,linkcolor=Mahogany,urlcolor=Plum,pagebackref]{hyperref}
\usepackage{cleveref}
\usepackage[alphabetic]{amsrefs}

\usepackage{enumerate}
\usepackage{tikz}
\usepackage{tikz-cd}
\usepackage{verbatim,color,geometry}
\usepackage[all]{xy}
\usepackage{enumitem}
\usepackage{bm}
\usepackage{stmaryrd}
\usepackage{textcomp}
\usepackage{tablefootnote}

\newcommand{\cE}{\mathcal{E}}
\newcommand{\cF}{\mathcal{F}}
\newcommand{\cG}{\mathcal{G}}

\newcommand{\cJ}{\mathcal{J}}

\newcommand{\cN}{\mathcal{N}}
\newcommand{\cO}{\mathcal{O}}

\newcommand{\cT}{\mathcal{T}}

\newcommand{\fa}{\mathfrak{a}}

\newcommand{\fm}{\mathfrak{m}}

\newcommand{\bP}{\mathbb{P}}
\newcommand{\bC}{\mathbb{C}}
\newcommand{\bR}{\mathbb{R}}
\newcommand{\bA}{\mathbb{A}}
\newcommand{\bQ}{\mathbb{Q}}
\newcommand{\bZ}{\mathbb{Z}}

\newcommand{\bG}{\mathbb{G}}
\newcommand{\bF}{\mathbb{F}}

\newcommand{\bN}{\mathbb{N}}
\newcommand{\bk}{\mathbbm{k}}

\newcommand{\hvol}{\widehat{\mathrm{vol}}}

\newcommand{\red}{\mathrm{red}}
\newcommand{\cI}{\mathcal{I}}
\newcommand{\Bl}{\mathrm{Bl}}

\newcommand{\tY}{\widetilde{Y}}

\DeclareMathOperator{\Spec}{Spec}
\DeclareMathOperator{\vol}{vol}
\DeclareMathOperator{\mult}{mult}

\DeclareMathOperator{\Pic}{Pic}

\DeclareMathOperator{\ord}{ord}

\DeclareMathOperator{\sing}{sing}
\DeclareMathOperator{\Supp}{Supp}

\DeclareMathOperator{\Val}{Val}
\DeclareMathAlphabet{\mathbbb}{U}{bbold}{m}{n}

\newcommand{\mld}{\mathrm{mld}}

\newcommand{\tC}{\widetilde{C}}
\newcommand{\tE}{\widetilde{E}}
\newcommand{\tH}{\widetilde{H}}

\newcommand{\tx}{\tilde{x}}
\newcommand{\tX}{\widetilde{X}}

\newcommand{\hX}{\widehat{X}}

\newcommand{\crex}{\mathrm{crex}}
\newcommand{\oE}{\overline{E}}
\newcommand{\edim}{\mathrm{edim}}
\newcommand{\hx}{\hat{x}}
\newcommand{\ty}{\tilde{y}}
\newcommand{\hVol}{\widehat{\mathrm{Vol}}}
\newcommand{\bfw}{\mathbf{w}}
\newcommand{\tZ}{\widetilde{Z}}
\newcommand{\GL}{\mathrm{GL}}
\newcommand{\hW}{\widehat{W}}
\newcommand{\wsum}{\mathbf{w}_{\mathrm{s}}}
\newcommand{\wprod}{\mathbf{w}_{\mathrm{p}}}
\newcommand{\an}{\mathrm{an}}

\newcommand{\YL}[1]{{\textcolor{blue}{[Yuchen: #1]}}}

\numberwithin{equation}{section}

\newtheorem{prop} {Proposition} [section]
\newtheorem{thm}[prop] {Theorem} 

\newtheorem{lem}[prop] {Lemma}

\newtheorem{prop-def}[prop]{Proposition-Definition}

\theoremstyle{definition}
 
\newtheorem{rem}[prop] {Remark} 

\newtheorem{defn}[prop]{Definition}
\newtheorem{que}[prop]{Question}
\newtheorem{expl}[prop] {Example}

\title{Optimal bounds for local volumes of  threefold singularities}
\author{Yuchen Liu}
\address{Department of Mathematics, Northwestern University, Evanston, IL 60208, USA.}
\email{yuchenl@northwestern.edu}
\date{\today}

\begin{document}

\begin{abstract}
    We establish an optimal upper bound for local volumes of Gorenstein canonical non-hypersurface threefold singularities. Specifically, we show that a klt threefold  singularity with local volume at least $9$ is either a hypersurface singularity or a quotient singularity. As applications, we obtain new restrictions on the singularities of members in K-moduli spaces of Fano threefolds, and we establish a sharp inequality between local volumes and minimal log discrepancies for threefold singularities.
\end{abstract}

\maketitle

\section{Introduction}

The local volume $\hvol(x,X)$ of a klt singularity $x\in X$, introduced by Chi Li in \cite{Li18} as the minimum of the normalized volume of valuations, is a numerical invariant that has played a crucial role in the recent development of a local K-stability theory for singularities; see \cite{LLX18, Zhu25} for recent surveys. It is thus natural to further investigate  the distribution of local volumes in a fixed dimension $n$, such as explicit bounds and gaps. In \cite{LX19} it was shown that $\hvol(x,X)\leq n^n$ and equality holds if and only if $x\in X$ is smooth. Then it was conjectured in \cite{SS17}, known as the ODP Gap Conjecture, that the second largest local volume is $2(n-1)^n$ which is obtained by the ordinary double point. This conjecture is confirmed in dimension at most $3$ \cite{LL19, LX19} and for local complete intersection singularities \cite{Liu22}. More recently, it was shown in \cite{XZ24}, as a consequence of their local boundedness result for K-semistable Fano cone singularities, that the local volumes in a fixed dimension are discrete away from zero. See also \cite{HLQ23, Zhu24, LMS23} for previous results in lower dimensions and \cite{HLQ24} for general coefficient sets.

In this article, we establish an optimal upper bound for local volumes of Gorenstein canonical non-hypersurface threefold singularities. Our main result goes as follows.

\begin{thm}\label{thm:nv-can}
Let $x\in X$ be a Gorenstein canonical non-hypersurface threefold singularity. Then $\hvol(x,X)\leq 9$, and equality holds if and only if $x\in X$ is a cyclic quotient singularity  of type $\frac{1}{3}(1,1,1)$. 
\end{thm}

Theorem \ref{thm:nv-can} can be thought of as a refinement of the ODP Gap Theorem for threefolds in \cite{LX19} that any non-smooth threefold singularity $x\in X$ satisfies $\hvol(x,X)\leq 16$, and equality holds if and only if $x\in X$ is an $A_1$-singularity.

As a consequence, we have the following characterization of threefold singularities with local volumes at least $9$. See Example \ref{expl:ADE} for many examples of such singularities and their local volumes, and Question \ref{que:vol-list} for a conjectural list of all possible local volumes that are at least $9$.

\begin{thm}\label{thm:nv-3fold}
Let $x\in X$ be a klt threefold singularity. Then $\hvol(x,X)\geq 9$ if and only if  $x\in X$ is either a hypersurface singularity of type $cA_{\leq 2}$, or a cyclic quotient singularity of type $\frac{1}{2}(1,1,1)$, $\frac{1}{3}(1,1,0)$, $\frac{1}{3}(1,1,1)$, or $\frac{1}{3}(1,1,2)$. 
\end{thm}

The above results have some notable applications to the study of K-moduli of Fano threefolds through the local-to-global volume comparison \cite{Liu18}. See Section \ref{sec:3fold-sing} for a description of singularities in the following theorem.

\begin{thm}\label{thm:K-moduli}
Let $X$ be a K-semistable $\bQ$-Fano threefold with volume $(-K_X)^3 = V$.
\begin{enumerate}
    \item If $V\geq 26$, then $X$ has only $cA_1$-singularities or  cyclic quotient singularities of type $\frac{1}{2}(1,1,1)$.
    \item If $V\geq 22$, then $X$ has only $cA_1$-singularities, isolated $cA_2$-singularities, $D_\infty$-singularities, or cyclic quotient singularities of type $\frac{1}{2}(1,1,1)$.
    \item If $V\geq 11$, then any non-Gorenstein singularity $x\in X$ is a cyclic quotient of a (possibly smooth) hypersurface singularity of type $cA_{\leq 2}$.
\end{enumerate}
If in addition $X$ is $\bQ$-Gorenstein smoothable, then it does not admit cyclic quotient singularities of type $\frac{1}{2}(1,1,1)$.
\end{thm}

Theorem \ref{thm:K-moduli} can be viewed as a strengthening of a result proven in \cite{LX19, LZ24} stating that $X$ is Gorenstein canonical if $V\geq 20$. Such  restrictions on singularities played a key role in the description of K-moduli spaces for various families of Fano varieties via the moduli continuity method, see e.g.\ \cite{MM93, OSS16, SS17, LX19, Liu22, ADL21, LZ24, Zha24}. We expect Theorem \ref{thm:K-moduli} to be useful in the future study of K-moduli spaces for Fano threefolds. Specifically, according to \cite{ACC+}, there are precisely $79$ out of $105$ families of smooth Fano threefolds that contain K-semistable members, i.e.\ whose K-moduli spaces are non-empty. Among these $79$ families, Theorem \ref{thm:K-moduli} (1), (2), and (3) can be applied to $35$, $46$, and $68$ families, respectively, where   $3$, $9$, and $28$ families respectively have unknown K-moduli compactifications to the author's knowledge. See Remark \ref{rem:K-mod} for a list of these families.

As another application, we establish a sharp comparison between the local volume and the minimal log discrepancy (mld) for klt threefold singularities, answering  \cite[Question 6.16]{LLX18} affirmatively in dimension $3$. 

\begin{thm}\label{thm:nv-mld}
Let $x\in X$ be a klt threefold  singularity. Then 
\begin{equation}\label{eq:nv-mld}
    \hvol(x,X) \leq 9 \cdot \mld(x,X).
\end{equation}
Moreover, equality holds if and only if $x\in X$ is a cyclic quotient singularity of type $\frac{1}{r}(1,1,1)$ for some $r\in \bZ_{>0}$.
\end{thm}

Note that a weaker inequality $\hvol(x,X)< n^n\cdot \mld(x,X)$ was proven in any dimension $n$ in \cite[Theorem 6.13]{LLX18}.

Let us briefly explain the strategy of the proof of Theorems \ref{thm:nv-can} and \ref{thm:nv-3fold}, which is similar to \cite{LX19} but technically more involved. For hypersurface singularities, we use local equations and monomial valuations on the ambient space to get upper bounds for its local volumes.  Suppose $x\in X$ is a Gorenstein canonical non-hypersurface threefold singularity. Then the hope is to either find a good divisor with small normalized volume or to bound local volumes on a crepant birational model. This is done by running a suitable MMP on the terminalization to contract crepant exceptional divisors one by one. Compared to \cite{LX19}, the new difficulty lies in the fact that in order to achieve the local volume bound $9$ ($=$ local volume of a transversal $A_2$-singularity), we need to analyze the behavior of contractions of two exceptional divisors, as contracting only one divisor may create a transversal $A_1$-singularity whose local volume is $\frac{27}{2}$, which is not strong enough for our estimate. Then we reduce to the case when there are at most two crepant exceptional divisors over $x\in X$ and apply the detailed analysis  from \cite{Lau77, Rei76, Rei80, Rei94}. In some cases we are only able to get a good crepant birational model after an \'etale base change using Artin approximation theorem (see Section \ref{sec:Artin}), a result that is of independent interest.

\subsection*{Acknowledgements} The author would like to thank Anne-Sophie Kaloghiros, J\'anos Koll\'ar, Andrea Petracci, Jakub Witaszek, Chenyang Xu, and Junyan Zhao for helpful discussions and comments. The author is partially supported by NSF CAREER Grant DMS-2237139 and an AT\&T Research Fellowship from Northwestern University.

\section{Preliminaries}

Throughout, we work over an algebraically closed field $\bk$ of characteristic zero. We follow standard notation and conventions from \cite{KM98, Kol13}. A \emph{pair} $(X, \Delta)$ is a normal variety $X$ together with an effective $\bQ$-divisor $\Delta$ such that $K_X+\Delta$ is $\bQ$-Cartier. A \emph{singularity} $x\in X$ is a normal variety $X$ together with a closed point $x\in X$. A \emph{klt singularity} $x\in (X,\Delta)$ is a pair $(X,\Delta)$ and a closed point $x\in X$ such that $(X,\Delta)$ is klt in a neighborhood of $x$. A \emph{klt log Fano pair} $(X,\Delta)$ is a projective klt pair such that $-K_X-\Delta$ is ample. We say $X$ is a \emph{$\bQ$-Fano variety} if $(X,0)$ is a klt log Fano pair.

\subsection{Local volumes}

The normalized volume functional and its minimum, the local volume, were introduced in \cite{Li18}. 
We refer to \cite{LLX18} for the common terminology of normalized volumes and local volumes. We will only recall a few things that are repeatedly used in this paper.

\begin{defn}
Let $X$ be a normal variety.
A \emph{divisor $E$ over $X$} is a prime divisor $E$ on a normal birational model $Y\to X$. For a singularity $x\in X$, a divisor $E$ over $X$ whose center is $x$ is called a \emph{divisor over $x\in X$}.
\end{defn}

\begin{defn}
Let $x\in X$ be a klt singularity of dimension $n$.
We denote by $\Val_{X,x}$ the \emph{space of real valuations of $K(X)$ centered at $x$}. Let $A_X: \Val_{X,x} \to (0, +\infty]$ be the \emph{log discrepancy} function. Let $\vol_{X,x}: \Val_{X,x}\to [0, +\infty)$ be the \emph{volume} function. Denote by $\Val_{X,x}^{\circ}$ the subspace of $\Val_{X,x}$ consisting of valuations $v$ with $A_X(v)<+\infty$.  

We define the \emph{normalized volume} functional $\hvol_{X,x}: \Val_{X,x} \to (0, +\infty]$ as
\[
\hvol_{X,x}(v):= \begin{cases} A_X(v)^n \cdot \vol_{X,x}(v) & \textrm{ if }v\in \Val_{X,x}^{\circ}\\ +\infty & \textrm{ if }v\not\in \Val_{X,x}^{\circ}
\end{cases}
\]

The \emph{local volume} of a klt singularity $x\in X$ is defined as 
\[
\hvol(x, X) := \inf_{v\in \Val_{X,x}} \hvol_{X,x}(v).
\]
\end{defn}

The following theorem is at the center of the study for the normalized volume functional. Since we only need the uniqueness part of the theorem, we will avoid introducing the related notions such as Fano cone singularities.

\begin{thm}[{Stable Degeneration Theorem, \cite{Blu18, LX18, LWX18, XZ20,  BLQ24, XZ25}}]\label{thm:SDC}
Let $x\in X$ be a klt singularity. Then there exists a minimizer $v_*$ of $\hvol_{X,x}$ unique up to scaling. Moreover, $v_*$ is quasi-monomial with a finitely generated associated graded algebra and induces a degeneration of $x\in X$ to a K-semistable Fano cone singularity which further degenerates to a unique K-polystable Fano cone singularity.
\end{thm}

\begin{defn}
A projective birational morphism $f: Y\to X$ from a normal variety $Y$ is called an \emph{lc blow-up} if $f$ is isomorphic over $X\setminus \{x\}$, $f^{-1}(x)$ is the support of a reduced divisor $E$, $(Y,E)$ is lc, and $-K_Y-E$ is $\bQ$-Cartier and ample over $X$. We sometimes write $f:(Y,E)\to X$ to denote an lc blow-up. 

If in addition $(Y,E)$ is plt, we say that $f$ is a \emph{plt blow-up} extracting a \emph{Koll\'ar component} $E$ over $x\in X$. We sometimes say that $(E, \Delta_E)$ is a Koll\'ar component over $x\in X$ where $\Delta_E$ is the different divisor of $(Y,E)$ along $E$. 
\end{defn}

We note that if either $E$ is irreducible (e.g.\ when $f$ is a plt blow-up) or $f$ is crepant, then the $f$-ampleness of $-K_Y - E$ in the above definition is the same as the $f$-ampleness of $-E$: in the former case this is because $-K_Y-E \sim_{\bQ,f} -A_X(E) E$ with $A_X(E)>0$, while in the latter case $-K_Y-E \sim_{\bQ,f} -E$.

\begin{defn}
We say that two singularities $x\in X$ and $x'\in X'$ are \emph{analytically isomorphic} if we have $\widehat{\cO_{X,x}}\cong \widehat{\cO_{X',x'}}$ as $\bk$-algebras. 
\end{defn}

Note that our notion of ``analytically isomorphic'' refers to ``formally isomorphic'' in the literature. Nevertheless, a famous result of Artin \cite[Corollary 2.6]{Art69} shows that our notion of analytically isomorphic singularities have isomorphic \'etale neighborhoods, hence have isomorphic complex analytic germs when  $\bk = \bC$. We will often not distinguish between analytic isomorphisms and \'etale equivalences.

\begin{prop}[{\cite[Proposition 2.24]{HLQ23}}]\label{prop:nv-analytic}
Let $x\in X$ and $x'\in X'$ be analytically isomorphic klt singularities. Then $\hvol(x,X) = \hvol(x', X')$.
\end{prop}

\begin{expl}[\cite{LL19, LX20}] \label{expl:quotient}
Let $x\in X$ be an $n$-dimensional quotient singularity, i.e.\ it is  analytically isomorphic to $0\in \bA^n/G$ where $G\leq \GL(n,\bk)$ is a finite subgroup acting on $\bA^n$ effectively and freely in codimension $1$. Then we have $\hvol(x,X) = \frac{n^n}{|G|}$.
\end{expl}

\begin{thm}[\cite{XZ20}]\label{thm:finite-deg}
Let $\pi: (y\in Y) \to (x\in X)$ be a quasi-\'etale morphism between klt singularities. Then we have
\[
\hvol(y, Y) = \deg(\pi) \cdot \hvol(x,X).
\]
\end{thm}


\begin{thm}[{\cite[Corollary 2.12]{LX19}}]\label{thm:nv-crepant}
Let $\phi : (y\in Y) \to (x\in X)$ be a crepant birational morphism of klt singularities such
that $y \in \mathrm{Ex}(\phi)$. Then $\hvol( x, X) < \hvol( y, Y )$.
\end{thm}

\begin{thm}[\cite{LX19}]\label{thm:ODP-gap}
Let $x\in X$ be a klt singularity of dimension $n$. Then 
$\hvol(x,X) \leq n^n$, and equality holds if and only if $x\in X$ is smooth.

Moreover, if $n=3$ and $x\in X$ is not smooth, then $\hvol(x,X) \leq 16$, and equality holds if and only if $x\in X$ is an $A_1$-singularity. 
\end{thm}

\begin{defn}
Let $X$ be a klt variety and $\eta\in X$ a scheme-theoretic point. We define the \emph{minimal log discrepancy} of $\eta\in X$ as
\[
\mld(\eta, X):= \inf_{E} A_X(E),
\]
where $E$ runs over all divisors over $X$ centered at $\eta$.
\end{defn}

\subsection{Threefold hypersurface singularities}\label{sec:3fold-sing}

Throughout this subsection, let $x\in X$ be a threefold hypersurface  singularity, i.e.\ its embedding dimension is at most $4$. Note that a smooth point is a hypersurface singularity  in our convention.

\begin{defn}
We say that $x\in X$ is an \emph{$A_k$-singularity} for $k\geq 1$, a \emph{$D_k$-singularity} for $k\geq 4$, or an \emph{$E_k$-singularity} for $k\in \{6,7,8\}$ if locally analytically it is given by the equation $x_1 x_2 + f(x_3,x_4) = 0$ where $f$ is determined by the following: 
\[
f(x_3, x_4) = \begin{cases}
x_3^2 + x_4^{k+1} & A_k\\
x_3^2 x_4 + x_4^{k-1} & D_k\\
x_3^3 + x_4^4 & E_6\\
x_3^3 + x_3 x_4^3 & E_7\\
x_3^3 + x_4^5 & E_8
\end{cases}
\]
\end{defn}

\begin{defn}
We say that $x\in X$ is an \emph{$A_\infty$-singularity} (resp.\ a \emph{$D_\infty$-singularity}) if locally analytically it is given by the equation $x_1 x_2 + x_3^2 = 0$ (resp.\ $x_1x_2 + x_3^2 x_4 = 0$).
\end{defn}

\begin{defn}
We say that $x\in X$ is \emph{compound Du Val (cDV)} if a general hyperplane section $H$ through $x$ satisfies that $x\in H$ is a du Val surface singularity. Moreover, if $x\in H$ is of type $A_k$ ($k\geq 1$), $D_k$ ($k\geq 4$), or $E_k$ ($k\in \{6,7,8\}$) for a general $H$, we say that $x\in X$ is of type $cA_k$, $cD_k$, or $cE_k$, respectively. 
\end{defn}

By convention, we treat a smooth surface and a smooth threefold as a surface $A_0$-singularity and a $cA_0$-singularity, respectively.

It is easy to see that for threefolds, an $A_k$ or $A_\infty$-singularity is of type $cA_1$ (the converse also holds), while a $D_k$, $D_\infty$, or $E_k$-singularity is of type $cA_2$. 

\begin{defn}
We say that $x\in X$ is a \emph{transversal du Val singularity} if it is locally analytically isomorphic to $(s,0) \in S\times \bA^1$ for a du Val surface singularity $s\in S$. Moreover, we say $x\in X$ is a \emph{transversal $A_k$, $D_k$, or $E_k$-singularity} if $s\in S$ is a surface $A_k$, $D_k$, or $E_k$-singularity, respectively.
\end{defn}

Clearly, a transversal $A_k$, $D_k$, or $E_k$-singularity is of type $cA_k$, $cD_k$, or $cE_k$, respectively. 
It is easy to see that $A_\infty$-singularities are the same as transversal $A_1$-singularities.

\subsection{Crepant exceptional divisors}

\begin{defn}
    Let $x\in X$ be a klt singularity. Define $\crex(x, X)$ to be the number of prime divisors $E$ over $X$ centered at $x$ such that $A_X(E) = 1$. We call such $E$ a \emph{crepant exceptional divisor} over $x\in X$. Note that $\crex(x,X)$ is always finite by \cite[Proposition 2.36]{KM98}.
\end{defn}

\begin{prop}\label{prop:extract}
Let $x\in X$ be a klt singularity. Let $E$ be a prime divisor over $x\in X$ satisfying $A_X(E)\leq 1$. Then there exists a projective birational morphism $\mu: Y\to X$ from a normal variety $Y$ such that $\mu$ is an isomorphism over $X\setminus\{x\}$, $\mu^{-1}(x) = E$, and $-E$ is $\bQ$-Cartier and ample.
\end{prop}

\begin{proof}
    We can add a small effective $\bQ$-Cartier $\bQ$-divisor $\Delta$ through $x$ such that $(X, \Delta)$ is klt and $A_{X,\Delta}(E)<1$. Then the statement follows from the pair version of \cite[Proposition 1.5]{Blu21} which relies on \cite[Theorem 1.2]{BCHM10} (see also \cite[Corollary 1.39]{Kol13}).
\end{proof}

\begin{prop}\label{prop:crex-analytic}
Let $x\in X$ and $x'\in X'$ be analytically isomorphic klt singularities. Then we have $\crex(x,X) = \crex(x', X')$.
\end{prop}

\begin{proof}
Denote the crepant exceptional divisors over $x\in X$ by $\{E_1,\cdots, E_l\}$ with $l := \crex(x,X)$. Let $v_i:= \ord_{E_i}\in \Val_{X,x}^{\circ}$. By \cite[Proof of Proposition 2.24]{HLQ23}, we have a bijective map $\phi:\Val_{X,x}^{\circ} \to \Val_{X',x'}^{\circ}$ with $v_i':=\phi(v_i)$ such that $A_X(v_i) = A_{X'}(v_i')$. Moreover, by \cite[Lemma 3.10]{JM12} we know that $v_i'$ is also divisorial whose value group is $\bZ$. Thus there exist distinct divisors $\{E_i'\}_{1\leq i\leq l}$ over  $x'\in X'$ such that $v_i'= \ord_{E_i'}$ and $A_{X'}(E_i') = A_X(E_i) = 1$. Thus we have $\crex(x', X') \geq \crex(x,X)$. By symmetry we get the reversed inequality and hence the equality.
\end{proof}

\begin{thm}[{\cite[Theorems 5.34 and 5.35]{KM98}}]\label{thm:crex-cDV}
Let $x\in X$ be a Gorenstein canonical threefold singularity.
\begin{enumerate}
    \item We have $\crex(x,X)=0 $ if and only if $x\in X$ is a cDV hypersurface singularity.
    \item Suppose $x\in X$ is not a hypersurface singularity. Then we have $\crex(x, X) > 0$, the blow-up $Y = \Bl_x X \to X$ is a crepant birational morphism, and $Y$ is Gorenstein canonical. Moreover, the embedding dimension of $x\in X$ is its Hilbert--Samuel multiplicity plus $1$.
\end{enumerate}



\end{thm}

\begin{lem}[{\cite{Rei80}, \cite[Theorems 5.34 and 5.35]{KM98}}]\label{lem:can-elliptic}
Let $x\in X$ be a Gorenstein canonical non-hypersurface threefold singularity. Let $H\subset X$ be a general hyperplane section through $x$. Then $x\in H$ is a Gorenstein elliptic surface singularity.
\end{lem}

\subsection{Gorenstein elliptic surface singularities}

Throughout this subsection, let $x\in X$ be a Gorenstein elliptic surface singularity. Results in this subsection are mainly from \cite{Lau77}; see also \cite{Rei76} and \cite[Section 4.4]{KM98}.

\begin{thm}[{\cite[Proposition 3.5]{Lau77}}]\label{thm:laufer-resol}
 Let $f:\tX \to X$ be a minimal resolution.
 Let $E_1,\cdots, E_l$ be all the  exceptional curves of $f$. Then either $f$ is a log resolution with only rational exceptional curves, or one of the following five cases occurs. 
 \begin{enumerate}
     \item $l=1$ and $E_1$ is a smooth elliptic curve.
     \item $l=1$ and $E_1$ is a one-nodal rational curve.
     \item $l=1$ and $E_1$ is a one-cuspidal rational curve.
     \item $l = 2$ and $E_1+E_2$ is the sum of two smooth rational curves intersecting at a tacnode.
     \item $l=3$ and $E_1+E_2+E_3$ is the sum of three smooth rational curves intersecting at a common ordinary triple point.
 \end{enumerate}
\end{thm}

\begin{thm}[{\cite[Theorem 3.13]{Lau77}}]\label{thm:Laufer-blowup}
Denote by $e:=-(Z^2)$ where $Z$ is a fundamental cycle on some resolution of $x\in X$. If $e\leq 3$, then $x\in X$ is a hypersurface singularity. If $e \geq 3$, then $\Bl_x X\to X$ is the canonical modification. 
\end{thm}

Let $Z$ denote the fundamental cycle on a minimal resolution $f:\tX \to X$. Then according to \cite{Lau77}, we have  $\chi(Z,\cO_Z) = 0$ and $\chi(Z',\cO_{Z'}) > 0$ for any $0<Z'< Z$. Moreover, we have $K_{\tX} + Z = f^* K_X$. 


\begin{prop}\label{prop:fund-cycle-2}
    Suppose a minimal resolution $f:\tX \to X$ has precisely two exceptional curves $E_1$ and $E_2$. Then the fundamental cycle $Z =  E_1+E_2$ is reduced, both $E_1$ and $E_2$ are smooth rational curves, and $(E_1\cdot E_2) = 2$. In particular, the intersection $E_1\cap E_2$ is either transversal or tacnodal.
\end{prop}

\begin{proof}
By Theorem \ref{thm:laufer-resol}, we know that $f$ is either a log resolution with only rational exceptional curves or the  exceptional case (4). Thus we may assume that we are in the former case as (4) satisfies the statement.

Assume to the contrary that $Z$ is not reduced. Then we can take $Z'$ to be the reduced cycle whose support is the same as $Z$. Then we have $\chi(Z',\cO_{Z'})>0$ which implies that $Z'$ is a tree of two $\bP^1$'s, i.e.\ $(E_1\cdot E_2)=1$. Then $Z' = E_1+E_2$ is a fundamental cycle as it satisfies $(Z'\cdot E_i) = 1+(E_i)^2 <0$, a contradiction to $Z'<Z$. Thus $Z$ is reduced, both $E_1$ and $E_2$ are $\bP^1$'s, and $(E_1\cdot E_2)=2$ by the condition that $\chi(Z,\cO_Z) = 0$.
\end{proof}

\section{Factorizing contractions by Artin approximation}\label{sec:Artin}

The goal of this section is to factorize birational morphisms after an \'etale base change using Artin approximation theorem \cite{Art69} and vanishing theorems. 

\begin{prop}\label{prop:Pic-lc}
Let $x\in X$ be a klt singularity over $\bk= \bC$.
Let $f:(Y,E)\to X$ be an lc blow-up. 
Let $L_E$ be a line bundle over $E$.  Then there exists an \'etale morphism $(x'\in X')\to (x\in X)$ and a line bundle $L'$ on $Y':=Y\times_X X'$ such that $L'|_{E'} \cong L_E$ under the natural identification of $E':=E\times_X X'$ and $E$. Moreover, if $L_E$ is nef, then $L'$ is semiample over $X'$. 
\end{prop}

\begin{proof}
Denote by $(R,\fm) := (\cO_{X,x}, \fm_{X,x})$ and $Y_R:= Y\times_X \Spec R$. Let $R^h$ be the Henselization of $R$ and $Y_{R^h}:= Y\times_X \Spec R^h$. Denote by $X_i:= \Spec R/\fm^i$ and $Y_i:= Y\times_X X_i$. 
By \cite[Theorem 3.5]{Art69}, the restriction map $\Pic(Y_{R^h})\to \varprojlim_i \Pic(Y_i)$ has a dense image. Let $E_k$ be the closed subscheme of $Y$ with ideal sheaf $\cI_k=\cO_Y(-kE)$. Denote by $\cI:=f^{-1} \fm$. Since the cosupport of $\cI$ and $\cI_1$ are the same, we have $\sqrt{\cI} = \cI_1$ which implies that the direct systems $\{E_k\}_k$ and $\{Y_i\}_i$ are cofinal. Thus we have $\varprojlim_i \Pic(Y_i) = \varprojlim_k \Pic(E_k)$ and hence the restriction map $\Pic(Y_{R^h})\to \varprojlim_k \Pic(E_k)$ has a dense image.


Next, we claim that $\varprojlim_{k}\Pic(E_k)\to \Pic(E)$ is surjective. Let $X^{\an}$ be a sufficiently small Stein contractible complex analytic neighborhood of $x\in X$. Denote by $f^{\an}: (Y^{\an}, E^{\an})= (Y,E)\times_X X^{\an} \to X^{\an}$. Note that we equip complex analytic topology on spaces $X^{\an},  Y^{\an}, E^{\an}$, etc. 
If we write $E= \sum_{i} F_i$ and $A:=\sum_i A_X(F_i) F_i$, then $-A\sim_{\bQ,f} -K_Y-E$ is $\bQ$-Cartier and $f$-ample, and $A$ is fully supported on $E$. Since $(Y,E)$ is lc and $Y\setminus E\cong X\setminus \{x\}$ is klt, we know that $(Y, E-\epsilon A)$ is klt for $0<\epsilon \ll 1$. Thus both $X$ and $Y$ have rational singularities by \cite[Theorem 5.22]{KM98}. By Leray spectral sequence we know that $R^j f_* \cO_Y = 0$ for $j>0$. By relative GAGA, this implies  $R^j f^{\an}_*\cO_{Y^{\an}} = 0$. Since $X^{\mathrm{an}}$ is Stein, the Leray spectral sequence gives $H^j(Y^{\an}, \cO_{Y^{\an}}) = 0$ for $j>0$.
Thus applying the exponential sequence we get $\Pic(Y^{\an}) \cong H^2(Y^{\an},\bZ)$.
After shrinking $X^{\mathrm{an}}$ further if necessary, $E^{\mathrm{an}}$ is a deformation retract of $Y^{\mathrm{an}}$ as in \cite[Proof of Lemma 12.1.1]{KM92}. Thus
$H^2(Y^{\mathrm{an}}, \mathbb{Z})
\cong H^2(E^{\mathrm{an}}, \mathbb{Z})$.\footnote{This argument is suggested by J\'anos Koll\'ar.} 

Let $\Delta_E$ be the different $\bQ$-divisor of $(Y,E)$ along $E$. Then by adjunction we know that $-K_E-\Delta_E$ is ample and $(E,\Delta_E)$ is slc. Thus by Fujino's generalization of Kawamata--Viehweg vanishing theorem \cite[Theorem 1.7]{Fuj14} we know that $H^j(E,\cO_E) = 0$ for $j>0$ which implies $H^j(E^{\an}, \cO_{E^{\an}}) = 0$ by GAGA. Thus applying the exponential sequence again to $E^{\an}$ yields $\Pic(E^{\an})\cong H^2(E^{\an},\bZ)$. Combining the above isomorphisms, we conclude that the restriction map $\Pic(Y^{\an}) \to \Pic(E^{\an})$ is an isomorphism. In particular, $\varprojlim_k\Pic(E^{\an}_k) \to \Pic(E^{\an})$ is surjective. Thus the claim follows from GAGA.

By the claim and the dense image result we know that the restriction map $\Pic(Y_{R^h})\to \Pic(E)$ is surjective. Since $\Spec R^h$ is the inverse limit of $\Spec S$ where $S$ is a local \'etale $R$-algebra, we know that $\Pic(Y_{R^h})$ is the direct limit of $\Pic(Y_S)$ where $Y_S:= Y\times_X \Spec S$. Thus for any line bundle $L_E$ over $E$ there exists a local \'etale $R$-algebra $S$ and a line bundle $L_S$ over $Y_S$ such that $L_S|_{E} \cong L_E$. Since $S$ is essentially of finite type over $\bk$ with residue field $\bk$ as $\bk=\bC$ is algebraically closed, there exists an \'etale morphism $(x'\in X')\to (x\in X)$ between klt singularities such that $S = \cO_{X',x'}$ and  $L_S$ extends to a line bundle $L'$ on $Y'= Y\times_X X'$. Finally, if $L_E$ is nef, then $L'$ is nef over $X'$ as $Y'\to X'$ is isomorphic away from $x$.
Thus the semiampleness of $L'$ follows from the basepoint free theorem as $(Y,E-\epsilon A)$ is klt and $-K_Y - (E-\epsilon A)$ is ample over $X$ for $0<\epsilon\ll 1$. 
\end{proof}

\begin{prop}\label{prop:Pic-flop}
Let $x\in X$ be a klt singularity.
Let $f:Y\to X$ be a projective birational morphism from a normal variety $Y$ of dimension at least $3$. Suppose $f$ is an isomorphism over $X\setminus \{x\}$ and $f^{-1}(x)$ has dimension $1$. Let $C$ be an irreducible component of $f^{-1}(x)$. Then there exists an \'etale morphism $(x'\in X')\to (x\in X)$ with $Y':=Y\times_X X'$ and a projective birational morphism $g': Y'\to Z'$ over $X'$ to a normal variety $Z'$ such that $g'$ only contracts $C':= C\times_X X'$ and is isomorphic elsewhere.
\end{prop}

\begin{proof}
We use the same notation $(R,\fm), Y_R, R^h, Y_{R^h}, X_i, Y_i$ as in the first paragraph of the proof of Proposition \ref{prop:Pic-lc}. 
By \cite[Theorem 3.5]{Art69}, the restriction map $\Pic(Y_{R^h})\to \varprojlim_i \Pic(Y_i)$ has a dense image. 

Next, we show that if $\cJ_1\subset\cJ_2\subset \cO_Y$ are two ideal sheaves with $F_i = V(\cJ_i)$ such that $\Supp(F_1) = \Supp(F_2) = f^{-1}(x)$, then the restriction map $\Pic(F_1)\to \Pic(F_2)$ is surjective.  Indeed, since $\sqrt{\cJ_1} = \sqrt{\cJ_2}$, we know that $\cO_{F_1}^{\times}\to \cO_{F_2}^{\times}$ is surjective such that its kernel $\cF$ is a sheaf whose support has dimension $\leq 1$. Apply the long exact sequence of cohomology to the short exact sequence
\[
0\to \cF \to \cO_{F_1}^{\times}\to \cO_{F_2}^{\times}\to 1,
\]
we get 
\[
H^1(F_1, \cO_{F_1}^{\times}) \to H^1(F_2, \cO_{F_2}^{\times})\to H^2(Y, \cF) =0,
\]
where the vanishing of the last term follows from Grothendieck vanishing theorem  for
arbitrary sheaves of abelian groups \cite[\href{https://stacks.math.columbia.edu/tag/02UZ}{Tag 02UZ}]{stacksproject}. Thus  $\Pic(F_1)\to \Pic(F_2)$ is surjective.

In particular, the inverse system $\{\Pic(Y_i)\}_i$ has surjective maps. Hence $\Pic(Y_{R^h}) \to \Pic(Y_i)$ is surjective for every $i\in \bN$. Let $F$ be the reduced closed subscheme of $Y$ whose support is $f^{-1}(x)$. Then from earlier arguments we know that $\Pic(Y_i)\to \Pic(F)$  is surjective which implies that $\Pic(Y_{R^h}) \to \Pic(F)$ is surjective. 

Next, denote the irreducible components of $F$ by $C_1 = C, C_2, \cdots, C_l$. Let $L_{F}$ be a line bundle on $F$ such that $L_{F}|_C$ is trivial and $L_{F}|_{C_j}$ is ample for any $j\geq 2$. This can be done by choosing closed points $p_j\in C_j$ that are smooth in $F$ and setting $L_{F}:= \cO_{F}(\sum_{j=2}^l p_j)$. Since $\Pic(Y_{R^h}) \to \Pic(F)$ is surjective, by the same argument as the last paragraph of the proof of Proposition \ref{prop:Pic-lc}, there exists a local \'etale $R$-algebra $S$ and a line bundle $L_S$ over $Y_S$ such that $L_S|_{F} \cong L_F$. Thus we have an \'etale morphism $(x'\in X')\to (x\in X)$ where $S = \cO_{X',x'}$ and a line bundle $L'$ on $Y'=Y\times_X X'$ extending $L_S$ such that $L'|_{F'}\cong L_F$ under the natural identification of $F':=F\times_X X'$ and $F$. Since $L_F$ is nef, we know that $L'$ is nef over $X'$. Then by the assumption that $f$ is small, we have $K_Y = f^* K_X$ which implies that $K_{Y'} = {f'}^* K_{X'}$. Thus $L'-K_{Y'}$ is big and nef over $X'$ and $Y'$ has klt singularities. By the basepoint free theorem we have that $L'$ is semiample over $X'$ hence defines a projective birational morphism $g':Y'\to Z'$. Since $L_F$ is ample on every $C_j$ for $j\geq 2$ but trivial on $C$, we conclude that $g'$ only contracts $C'$. The proof is finished.
\end{proof}

The following proposition is suggested to the author by J\'anos Koll\'ar showing that  Proposition \ref{prop:Pic-flop} holds under a weaker assumption.

\begin{prop}
    Let $x\in X$ be a normal singularity. Let $f: Y \to X$ be a projective birational morphism from a normal variety $Y$. Assume that $f^{-1}(x)$ has a $1$-dimensional irreducible component $C$, and $R^1f_* \cO_Y = 0$. Then there exists an \'etale morphism $(x'\in X')\to (x\in X)$ with $Y':=Y\times_X X'\xrightarrow{f'} X'$ and a projective birational morphism $g': Y'\to Z'$ over $X'$ to a normal variety $Z'$ such that  $C':= C\times_X X'$ is the only curve in $f'^{-1}(x')$ contracted by $g'$.
\end{prop}    

\begin{proof}
Let $H\subset Y$ be a general very ample divisor. By Bertini's theorem, we may assume that $H$ is normal. Denote by $H\to W\to X$ the Stein factorization of $f|_H: H\to X$ where $W=\Spec_{X} (f|_{H})_*\cO_H$. Then we know that $\phi:H\to W$ has connected fibers and  $\psi: W\to X$ is finite. Let $(R,\fm):= (\cO_{X,x}, \fm_{X,x})$ and $S:= \psi_* \cO_W \otimes_{\cO_X} R$. Applying \cite[\href{https://stacks.math.columbia.edu/tag/09XF}{Tag 09XF} and \href{https://stacks.math.columbia.edu/tag/07M4}{Tag 07M4}]{stacksproject} to the finite ring map $R\to S$, we know that there exists an \'etale local $R$-algebra $R'$ with $S':= S\otimes_R R'$ such that the map $S'\to S'/\fm S'$ induces a bijection on idempotents. Since $R=\cO_{X,x}$ is essentially of finite type over $\bk$, there exists an \'etale  morphism $(x'\in X')\to (x\in X)$ between singularities with $R' \cong \cO_{X', x'}$ such that the connected components of $H':= H\times_X X'$ has a one-to-one correspondence with the connected components of $H'\cap f'^{-1}(x')$.

Let $H_1'$ be the sum of the connected components of $H'$ that meet $C'$, and $H_2':=H'-H_1'$. We claim that $H_2'$ is relatively basepoint free over $X'$, i.e.\ $f'^*f'_*\cO_{Y'}(H_2') \to \cO_{Y'}(H_2')$ is surjective. Since $H'|_{H_2'}=H_2'|_{H_2'}$ and $H'$ is relatively very ample over $X'$, it suffices to show that $f'_* \cO_{Y'}(H_2') \to (f'|_{H_2'})_* \cO_{H_2'}(H_2'|_{H_2'})$ is surjective. This follows from the long exact sequence of cohomology for the short exact sequence
\[
0\to \cO_{Y'} \to \cO_{Y'}(H_2')\to \cO_{H_2'}(H_2'|_{H_2'}) \to 0
\]
and $R^1f'_*\cO_{Y'} = 0$ (which follows from cohomology and base change and the assumption that $R^1 f_* \cO_Y = 0$). 

Finally, by the claim we have a relative ample model $g':Y'\to Z'$ of $H_2'$ over $X'$ which only contracts $C'$ in $f'^{-1}(x')$. 
\end{proof}

\section{Proofs}

In this section, we prove Theorems \ref{thm:nv-can} and \ref{thm:nv-3fold}. Since the local volume is invariant under (algebraically closed) field extensions by \cite[Proposition 14]{BL18a} and \cite[Corollary 1.2]{XZ20}, we may assume  $\bk = \bC$ throughout this section. 

\subsection{Hypersurface singularities}

We start with an estimate on the local volumes for general hypersurface singularities by monomial valuations.

\begin{prop}[{cf.\ \cite[Example 5.3]{Li18}}]\label{prop:hyp-weighted} Let $n\geq 2$ be an integer.
Let $x\in X$ be an $n$-dimensional klt hypersurface singularity where $\widehat{\cO_{X,x}}\cong \bk\llbracket  x_1,\cdots, x_{n+1}\rrbracket/ (f)  $. For  $\bfw = (w_1,\cdots, w_{n+1})\in \bR_{>0}^{n+1}$, denote by $v_{\bfw}$ the monomial valuation with weight $\bfw$ on $\bk\llbracket  x_1,\cdots, x_{n+1}\rrbracket$. Then for any $\bfw\in \bR_{>0}^{n+1}$ we have 
\begin{equation}\label{eq:hyp-weighted}
\hvol(x, X) \leq \left(\sum_{i=1}^{n+1} w_i - v_{\bfw}(f)\right)^n \cdot \frac{ v_{\bfw}(f)}{\prod_{i=1}^{n+1}w_i}. 
\end{equation}
\end{prop}

\begin{proof}
For simplicity, denote by $\wsum:=\sum_{i=1}^{n+1} w_i$ and $\wprod:=\prod_{i=1}^{n+1} w_i$.
It is clear that the right hand side of \eqref{eq:hyp-weighted} is a continuous function in $\bfw$ that is invariant under rescaling. Thus we may assume that $\bfw\in \bZ_{>0}^{n+1}$ and each pair $w_i,w_j$ with $i\neq j$ are coprime. We embed $X$ into a smooth variety $W$ of dimension $n+1$ such that $(x_1,\cdots, x_{n+1})$ is a local analytic coordinates. 
Let $\mu_W:\hW\to W$ be the $\bfw$-weighted blow-up extracting a unique prime divisor $F$. Then we know that $F\cong \bP(\bfw)$ and $-(-F)^{n+1} = \wprod^{-1}$. Moreover, from the assumption of pairwise coprimeness of $w_i$ we know that $\hW_{\sing}$ is finite. Thus $F$ is Cartier away from finitely many points. Let $Y$ be the strict transform of $X$ in $\hW$. Then we have $\mu= \mu_W|_Y: Y\to X$ a projective birational morphism. Denote by $Y'$ the normalization of $Y$ and the composition $\mu': Y' \to X$.  Then we have 
\[
K_{\hW} = \mu_W^* K_W + (\wsum - 1) F.  
\]
Clearly, we have $Y = \mu_W^* X - v_{\bfw}(f) F$. Denote by $F_{Y'}:=F|_{Y'}$. Since $F$ is Cartier away from a finite set, we know that $F_{Y'}$ is an effective $\bQ$-Cartier Weil divisor fully supported on the exceptional locus of $\mu'$.
By adjunction \cite[Proposition 4.5]{Kol13} there exists an effective $\bQ$-divisor $\Delta'$ on $Y'$ such that 
\[
K_{Y'} + \Delta' = (K_{\hW} + Y)|_{Y'} =  (\mu_W^* K_W + (\wsum - 1) F + Y)|_{Y'} = \mu'^* K_X + (\wsum - v_{\bfw}(f) - 1) F_{Y'}.
\]
Let $E$ be an irreducible component of $F_{Y'}$ with coefficient $b \geq 1$ as $F_{Y'}$ is an effective $\bZ$-divisor. Then we have 
\[
A_X(\ord_E) = 1+ (\wsum - v_{\bfw}(f) - 1) \ord_E(F_{Y'}) - \ord_E(\Delta') \leq  (\wsum - v_{\bfw}(f)) b.
\]
Moreover, we have  $\fa_{m}(\ord_E) \supset \mu'_* \cO_Y( - \frac{m}{b} F_{Y'})$ which implies that 
\[
\vol_{X,x}(\ord_E) \leq  -(-b^{-1} F_Y')^n = b^{-n}(- (-F)^{n}\cdot Y) = b^{-n}v_{\bfw}(f) (-(-F)^{n+1}) = b^{-n} v_{\bfw}(f) \wprod^{-1}. 
\]
As a result, we have
\[
\hvol(x, X) \leq A_X(\ord_E)^n \cdot \vol_{X,x}(\ord_E)  \leq \left(\wsum - v_{\bfw}(f)\right)^n \cdot \tfrac{ v_{\bfw}(f)}{\wprod}.
\]
The proof is finished.
\end{proof}

Next, we apply the above result to  hypersurface threefold singularities.

\begin{prop}\label{prop:nv-hyp}
Let $x\in X$ be a non-smooth hypersurface threefold singularity. If $\hvol(x,X) > 8$ then $x\in X$ is of type $cA_1$ or $cA_2$, in particular, $\crex(x,X) = 0$. Conversely, any $cA_1$ or $cA_2$ singularity satisfies $\hvol(x,X) \geq 9$.
\end{prop}

\begin{proof}
Firstly, if $x\in X$ has multiplicity at least $3$, then $\hvol(x,X) \leq 3$ by applying Proposition \ref{prop:hyp-weighted} to $\bfw = (1,1,1,1)$. If $x\in X$ has multiplicity $2$ but is not of $cA$-type, then we know that locally analytically $X = V( x_1^2 + f(x_2, x_3, x_4))$ with $\mult_0 f \geq 3$. 
Then by setting $\bfw = (3,2,2,2)$ in Proposition \ref{prop:hyp-weighted} we get 
\[
\hvol(x,X) \leq 3^3 \cdot \frac{6}{3\cdot 2\cdot 2\cdot 2} = \frac{27}{4}< 9.
\]
This shows that $x\in X$ has to be of $cA$-type. In suitable local analytic coordinates we write $X = V(x_1x_2 + g(x_3, x_4))$. If $x\in X$ is of type $cA_{\geq 3}$, then $\mult_0 g\geq 4$. 
Then by setting $\bfw = (2,2,1,1)$ in Proposition \ref{prop:hyp-weighted} we get 
\[
\hvol(x,X) \leq 2^3 \cdot \frac{4}{2\cdot 2\cdot 1\cdot 1} = 8.
\]

For the converse, one proof follows from the  lower semicontinuity of local volumes \cite{BL18a} and the fact that transversal $A_1$ or $A_2$-singularities are quotient singularities hence have local volume $\frac{27}{2}$ or $9$ (see Example \ref{expl:quotient}). Another proof follows from the adjunction type estimate \cite[Theorem 1.7]{LZ18} by taking general hyperplane sections.
\end{proof}

\begin{prop}\label{prop:nv-cA2}
Suppose $x\in X$ is a hypersurface threefold singularity of type $cA_2$. Then $\hvol(x,X) \leq \frac{32}{3}$.
\end{prop}

\begin{proof}
    We write $X= V(x_1 x_2 + f(x_3, x_4))$ in suitable local analytic coordinates such that $\mult_0 f = 3$. Then by setting $\bfw = (3,3,2,2)$ in Proposition \ref{prop:hyp-weighted} we get
    \[
    \hvol(x,X) \leq 4^3 \cdot \frac{6}{3\cdot 3\cdot 2 \cdot 2} = \frac{32}{3}.
    \]
\end{proof}

\begin{prop}\label{prop:non-iso-hyp}
Suppose $x\in X$ is a non-isolated hypersurface threefold singularity such that $\hvol(x, X) > 9$. Then $x\in X$ is either a transversal $A_1$-singularity or a $D_\infty$-singularity.
\end{prop}

\begin{proof}
By Proposition \ref{prop:nv-hyp} we know that $x\in X$ is of type $cA_1$ or $cA_2$. A non-isolated  $cA_1$-singularity is a transversal $A_1$-singularity. Thus it suffices to show that a non-isolated $cA_2$-singularity $x\in X$ with $\hvol(x,X)>9$ is of type $D_\infty$. Write $X = V(x_1 x_2 + f(x_3,x_4))$ in a local analytic coordinate. It is clear that the singular locus of $X$ is the same as the singular locus of $C := V(f(x_3,x_4))\subset \bC^2$. Since $x\in X$ is non-isolated, we know that $C$ is non-reduced. Moreover, since $X$ has type $cA_2$, we have $\mult_0 f = 3$. Thus an analytic non-reduced irreducible component $C_1$ of $C$ has multiplicity $1$ at $0$, which implies that under suitable local analytic coordinates we have $f(x_3,x_4) = x_3^2 g(x_3,x_4)$. Moreover, the fact that $\mult_0 f = 3$ implies $\mult_0 g = 1$, hence $V(g)$ defines a smooth plane curve at $0$. Denote by $k\in \bZ_{>0}$ the local intersection number of $V(g)$ and $V(x_3)$ at $0$. Then again under suitable local analytic coordinates we can write $f(x_3,x_4) = x_3^2 (x_3+x_4^k)$. We shall show that $k=1$ which implies that $x\in X$ is a $D_\infty$-singularity after a further change of coordinates $(x_3', x_4') := (x_3, x_3+x_4)$. Assume to the contrary that $k\geq 2$. Then by taking $\bfw =(3,3,2,1)$ in Proposition \ref{prop:hyp-weighted} we get $v_{\bfw}(x_1 x_2+f) = 6$ and 
\[
\hvol(x, X) \leq 3^3 \cdot \frac{6}{3\cdot 3 \cdot 2\cdot 1} = 9,
\]
which is a contradiction.
\end{proof}

\subsection{Cases when $\crex=1$}

\begin{prop}\label{prop:crex=1}
Theorem \ref{thm:nv-can} holds if $\crex(x,X) = 1$. 
\end{prop}

\begin{proof}

%
Let $E$ be the unique crepant exceptional divisor over $x\in X$.  Thus by Proposition \ref{prop:extract} we can extract $E$ on a crepant birational model $\mu:Y\to X$ such that $Y$ is also Gorenstein canonical and $E\subset Y$ is $\bQ$-Cartier, $\Supp(E) = \mu^{-1}(x)$, and $-E$ is ample over $X$. Thus for any $y\in E$ we have $\crex(y,Y)<\crex(x,X)$ which implies that $\crex(y,Y)=0$, i.e.\ $Y$ has cDV singularities along $E$ by Theorem \ref{thm:crex-cDV}.
Moreover, since there are no crepant exceptional divisor over $Y$ whose center is contained in $E$, we know that $Y_{\sing}\cap E$ has dimension $0$. 

We shall show that $E$ is Cartier in $Y$. If not, say $E$ is not Cartier at some point $y\in Y$ with Cartier index $d>1$.  Let $(\ty\in \tY)\to (y\in Y)$ be an index $1$ cover of $E$ at $y$. By Theorems \ref{thm:finite-deg} and \ref{thm:nv-crepant} we know that 
\[
\tfrac{1}{d}\hvol(\ty, \tY)= \hvol(y,Y)>\hvol(x, X) .
\]
If either $\ty\in \tY$ is singular or $d\geq 3$, then by Theorem \ref{thm:ODP-gap} we have $\hvol(x,X)<9$. If $\ty\in \tY$ is smooth and $d=2$, then $Y$ being Gorenstein implies that $y\in Y$ is a cyclic quotient singularity of type $\frac{1}{2}(1,1,0)$ whose singular curve is denoted by $C$. Thus $\tY\to Y$ is ramified along $C$ which implies that $E$ is not Cartier along $C$, in particular $C\subset E$. This contradicts the fact that $Y_{\sing}\cap E$ is a finite set. Thus we have shown that $E$ is Cartier in $Y$  and  $Y_{\sing}\cap E$ is a finite set.


Next, we follow similar arguments to \cite[Section 3.1]{LX19}. By adjunction, we know that $-K_E = -(K_Y+E)|_E= -E|_E$ is ample, and $E$ is Gorenstein by the Cartierness of $E$. Thus $E$ is an integral  Gorenstein del Pezzo surface. 
Thus we can apply the classification from \cite{Rei94}.
Similar to \cite[Proof of Proposition 3.4]{LX19}, we have
\begin{equation}\label{eq:lc-blowup}
\hvol(x,X) \leq \hvol_{X,x}(\ord_E) = (E^3) =  (-K_E)^2 .
\end{equation}

If $E$ is normal, then by \cite[Theorem 1]{Bre80} we get that either $E$ is a projective cone over an elliptic curve with local embedding dimensions at most $4$ (as $Y$ has hypersurface singularities along $E$) which implies $(-K_E)^2 \leq 4$, or $E$ has du Val singularities which implies $(-K_E)^2 \leq 9$. Both cases imply that $\hvol(x,X) \leq 9$. Moreover, if equality holds, then $(-K_E)^2 = 9$ and hence $E\cong \bP^2$. Thus, $E$ is a Koll\'ar component over $x\in X$. By \cite[Section 2.4]{LX20} it induces an isotrivial degeneration of $x\in X$ to the affine cone over $E\cong \bP^2$ with polarization $-E|_E\cong \cO_{\bP^2}(3)$, which is a cyclic quotient singularity of type $\frac{1}{3}(1,1,1)$. By Schlessinger's rigidity theorem \cite[Theorem 3]{Sch71} we know that they must be analytically isomorphic, which implies that $x\in X$ is a cyclic quotient singularity of type $\frac{1}{3}(1,1,1)$. 

In the rest of the proof, we may assume that $E$ is non-normal. Then the only two cases with possibly $(-K_E)^2\geq 9$ and the embedding dimension of every point $y\in E$ at most $4$ are projections of $\bF_{m-2;1}$ and $\bF_{m-4;2}$ where $m = (-K_E)^2$. These are also the main cases from \cite[Section 3.1]{LX19} thus we will generalize arguments there to allow $Y$ having hypersurface singularities. We shall show that $\hvol(x,X)<9$ in both cases, thus finishing the proof.

\smallskip

\textbf{Case 1.} $E$ is a projection of $\bF_{m-4;2}$. 

Denote by $A$ the fiber and $B$ the negative section of $\bF_{m-4;2}$. Under the normalization $\bF_{m-4;2}\to E$ we have a double cover $B\to C$ where $C$ is the non-normal locus in $E$. Both $B$ and $C$ are isomorphic to $\bP^1$. Thus we have a morphism $g_E: E\to C$ which is the descent of the composition $\bF_{m-4;2}\to B\to C$ where the first map is the usual $\bP^1$-bundle projection. Thus $L_E:=g_E^*\cO_{C}(1)$ is a nef line bundle $L_E$ on $E$. Clearly, a fiber of $g_E$ is either the union of two intersecting lines or a double line.  

Since $E$ is slc, inversion of adjunction implies that $(Y,E)$ is lc.
Moreover, since $\mu$ is crepant and $-E$ is ample over $X$, we know that $\mu:(Y,E)\to X$ is an lc blow-up. Thus by Proposition \ref{prop:Pic-lc} there exists an \'etale morphism $(x'\in X')\to (x\in X)$ and a line bundle $L'$ on $Y'=Y\times_X X'$ semiample over $X'$ such that $L'|_{E'}\cong L_E$ under the natural identification of $E'=E\times_X X'$ and $E$. By taking the ample model of $L'$ over $X'$, we obtain a projective birational morphism $g' : Y'\to Z'$ over $X'$ that contracts $E'$ to a curve $B_{Z'}$ whose fibers are naturally identified with the fibers of $g_E$ which are connected. Since $Y\to X$ is crepant, so is $Y'\to X'$. Hence both $Y'\to Z'$ and $Z'\to X'$ are crepant. Thus $Z'$ has Gorenstein canonical singularities. For a general point $z\in B_{Z'}$ over $x$, the fiber of $Y'\to Z'$ over $z$ has two irreducible components. Moreover, since $Y_{\sing}\cap E$ is finite, we may assume that $Y'$ is smooth in a neighborhood of $g'^{-1}(z)$. 
Take a general surface germ $S \subset Z'$ through $z$ transverse to the curve $B_{Z'}$.
Then the induced morphism $
g'_S : S_Y := g'^{-1}(S) \to S$
is a crepant resolution of the surface singularity $z \in S$.
Since $Y'$ is smooth near $g'^{-1}(z)$, the surface $S_Y$ is smooth, and the exceptional locus of $g'_S$ consists of precisely two irreducible curves.
Hence $z \in S$ is a du Val singularity of type $A_2$.  Since $z$ is a general point
of $B_{Z'}$, the transverse singularity type of $Z'$ along $B_{Z'}$ is
determined by such a general transverse surface section $z\in S$. 
Therefore $z \in Z'$ is a transversal $A_2$-singularity which implies that $\hvol(z, Z')  =  9$. Since by Theorem \ref{thm:nv-crepant} and Proposition \ref{prop:nv-analytic} we have
\[
\hvol(x, X) = \hvol(x', X') < \hvol(z, Z') = 9.
\]

\smallskip

\textbf{Case 2.} $E$ is a projection of $\bF_{m-2;1}$.

We follow similar arguments as \cite[Proof of Proposition 3.4]{LX19}. 
Denote by $C$ the non-normal locus of $E$ and $\nu:\oE \to E$ the normalization. Then we have $\nu^{-1}(C) = A\cup B$ where $A$ is a fiber of the $\bP^1$-bundle $\oE \cong \bF_{m-2}$ and $B$ is the negative section. Here the gluing map identifies $A$ and $B$ and fixes the point $A\cap B$. Let $y:=\nu(A\cap B) \in C$ be the only point in $E$ that is neither smooth nor normal crossing. 
Denote by 
$\cE$ the image of $\cT_{\oE}|_A\oplus \cT_{\oE}|_{B} \to \cT_Y|_C$. Then we have two short exact sequences 
\[
0 \to \cT_C \to \cT_{\oE}|_A\oplus \cT_{\oE}|_{B}  \to \cE \to 0,
\]
\[
0 \to \cE \to \cT_Y|_C \to \cG\to 0.
\]
Let $y_1',\cdots, y_k'$ be all points in $C\setminus \{y\}$ where $Y$ is singular. Then we have $\Supp(\cG)\subset\{y, y_1',\cdots, y_k'\}$ as  $(Y,E)$ is normal crossing along $C^\circ:=C\setminus \{y, y_1',\cdots, y_k'\}$. Next we analyze the relation of $\det \cE$ and $\omega_Y|_C$. Clearly along $C^\circ$ the two line bundles are dual to each other. 

At $y$, under suitable local analytic coordinates $(x_1,\cdots, x_4)$ we know that $E$ is locally  defined by $V(x_3^2 + x_2^3 - x_1x_2x_3, x_4)$ and $C = V(x_2, x_3, x_4)$ by \cite[Proof of Proposition 3.4]{LX19}. Since $E\subset Y$, we can write $Y = V(f)$ at $y$ where $f = (x_3^2 + x_2^3 - x_1x_2x_3) h + x_4 g$ such that $x_4$ does not appear in $h$.
A local generator of $\omega_Y|_C$ at $y$ is given by $ f_{x_4} ^{-1} dx_1\wedge dx_2 \wedge dx_3$ where $f_{x_4} = \frac{\partial{f}}{\partial{x_4}}|_C = g|_C$. Since $C\not\subset Y_{\sing}$ we know that $g|_C\neq 0$. Let $l :=\ord_{y} (g|_C)\in \bZ_{\geq 0}$. Thus a local generator of $\omega_Y|_C$ at $y$ is $x_1^{-l} dx_1\wedge dx_2 \wedge d x_3$.
By \cite[Proof of Proposition 3.4]{LX19}, we have $\cE = \langle \partial_{x_1}, x_1\partial_{x_2}, x_1^2\partial_{x_3} \rangle$ which implies that a local generator of $\det \cE$ at $y$ is $\partial_{x_1}\wedge x_1\partial_{x_2} \wedge x_1^2\partial_{x_3} = x_1^3 \partial_{x_1}\wedge \partial_{x_2} \wedge \partial_{x_3} $. Thus the natural pairing of $\det\cE$ and $\omega_Y|_C$ leads to a duality between $(\det\cE) \otimes \cO_C((3-l) y)$ and $\omega_Y|_C $ at $y$. 

Similarly, at $y_i'$ suppose $E = V(x_2x_3, x_4)$ and $C = V(x_2, x_3, x_4)$. Then $Y = V(f_i')$ where $f_i'=x_2 x_3 h_i' + x_4 g_i'$ such that $x_4$ does not appear in $h_i'$. We also have $g_i'|_C\neq 0$ as otherwise $Y$ is singular along $C$. Moreover, $g_i'$ has to vanish at $y_i'$ as otherwise $Y$ is smooth at $y_i'$, a contradiction. Let $l_i':=\ord_{y_i'} (g_i'|_C)\in \bZ_{> 0}$. Then a local generator of $\omega_Y|_C$ at $y_i'$ is given by $x_1^{-l_i'} dx_1\wedge dx_2 \wedge dx_3$ as $f_{x_4}/x_1^{l_i'}$ is an invertible function at $y_i'$. Since $\oE = V(x_2, x_4) \sqcup V(x_3, x_4)$ at $y_i'$ where $A$ and $B$ are given by $V(x_2,x_3,x_4)$ in each component, we know that $\cE$ is the image of $\cT_{\oE}|_A$ plus the image of $\cT_{\oE}|_B$ at $y_i'$ which implies that 
\[
\cE =  \langle \partial_{x_1},  \partial_{x_3}\rangle  + \langle \partial_{x_1},  \partial_{x_2}\rangle =  \langle \partial_{x_1},  \partial_{x_2},  \partial_{x_3}\rangle.
\]
Thus  $\det \cE $ has a local generator $\partial_{x_1}\wedge \partial_{x_2} \wedge \partial_{x_3} $ at $y_i'$. Similar to the previous discussion at $y$, we have a duality between $(\det\cE) \otimes \cO_C(-l_i' y_i')$ and $\omega_Y|_C$ at $y_i'$. To summarize, we have $(\det\cE) \otimes \cO_C ( (3-l) y - \sum_{i=1}^k l_i' y_i')\cong \omega_Y^{\vee}|_C$ which implies that 
\begin{align*}
 0 & = \deg \omega_Y^{\vee}|_C = \deg (\det\cE) \otimes \cO_C ( (3-l) y - \sum_{i=1}^k l_i' y_i')\\
 & = \deg \cE + 3-l - \sum_{i=1}^k l_i' \\
 &\leq \deg \cE + 3 = -\deg \cT_C +\deg \cT_{\oE}|_A + \deg \cT_{\oE}|_{B} + 3\\
 & = \deg \cT_C + \deg \cN_{A/\oE} + \deg \cN_{B/\oE} + 3\\
 & = 2 - (m-2) + 3  = 7-m.
\end{align*}
Thus we conclude from \eqref{eq:lc-blowup} that $\hvol(x,X) \leq m\leq 7$. 
\end{proof}

\subsection{Cases when $\crex = 2$}

\begin{prop} \label{prop:crex=2}
Theorem \ref{thm:nv-can} holds if $\crex(x,X) = 2$.
\end{prop}

The proof of Proposition \ref{prop:crex=2} is quite lengthy and is based on a thorough study of blow-ups of such singularities. We first fix some notation and derive some easy consequences.

Let $\mu:Y \to X$ be the blow-up of $\fm_{X,x}$. Denote by $e$ the Hilbert--Samuel multiplicity of $x\in X$. By Theorem \ref{thm:crex-cDV} we know that $e+1$ is the embedding dimension of $x\in X$, $\mu$ is crepant, and $Y$ has Gorenstein canonical singularities. Thus $e\geq 4$ since $x\in X$ is not a hypersurface singularity.
If $Y$ has a non-cDV singularity at some point $y$ over $x$, we have $0<\crex(y,Y)<\crex(x,X) = 2$ and hence $\crex(y,Y) = 1$. Thus by Theorem \ref{thm:nv-crepant} and Proposition \ref{prop:crex=1} we have 
\[
\hvol(x,X) < \hvol(y, Y) \leq 9.
\]
Thus we may assume that $Y$ has cDV singularities. Since $\crex(x,X)=2$, we know that $\mu$ extracts either one or two prime divisors. 

Below, we split the proof into three parts based on the behavior of $\mu$.

\begin{prop}\label{prop:crex=2-1}
    Theorem \ref{thm:nv-can} holds if $\crex(x,X) = 2$, the blow-up $\mu:Y\to X$ extracts only one prime divisor $E$, and $E$ is Cartier in $Y$.
\end{prop}

\begin{proof}
Since $E$ is Cartier in $Y$ and $-E$ is ample, we know that $E$ is an integral Gorenstein del Pezzo surface by adjunction.
Similar to the proof of Proposition \ref{prop:crex=1}, we only need to consider the cases where $E$ is a projection of $\bF_{m-2;1}$ or $\bF_{m-4;2}$. The $\bF_{m-4;2}$ case follows from Artin approximation results in Section \ref{sec:Artin} by the same arguments in the proof of  Proposition \ref{prop:crex=1}. Thus we may assume that $E$ is a projection of $\bF_{m-2;1}$. 

Let $C$ be the non-normal locus of $E$ and $\nu:\oE \to E$ the normalization. Then $\nu^{-1}(C)=A\cup B$ where $A$ is a fiber of $\oE \cong \bF_{m-2}$ and $B$ is the negative section. Let $y = \nu(A\cap B)\in C$. Since $E$ is Cartier in $Y$, we know that $Y$ is smooth along $E\setminus C$. By the assumption $\crex(y,Y) = 2$, we know that there is a crepant exceptional divisor $S$ over $Y$ whose center is contained in $E$. Since $Y$ has cDV singularities, we know that the center of $S$ cannot be a point by Theorem \ref{thm:crex-cDV}. Thus the center of $S$ in  $Y$ is the curve $C$ which implies that $Y$ is singular along $C$. We can assume that $\hvol(y', Y)>9$ for every point $y'\in C$, as otherwise we have $\hvol(x,X)< \hvol(y', Y)\leq 9$ by Theorem \ref{thm:nv-crepant}. Then Proposition \ref{prop:non-iso-hyp}  implies that $Y$ has either transversal $A_1$-singularities or $D_\infty$-singularities along $C$.

Next, let $\phi:\tY\to Y$ be the blow-up of $Y$ along $C$. By simple computations using the local equations of transversal $A_1$-singularities or $D_\infty$-singularities we know that $\phi$ is a log resolution of $Y$ extracting a smooth crepant exceptional  divisor $S$. Clearly, $\pi=\phi|_S: S\to  C$ is a conic bundle. Moreover, from the local equation of $E$ at $y$ we know that the strict transform $\phi_*^{-1}E$ is isomorphic to the normalization $\oE$, and we identify them for simplicity. 
Then $A+B = \oE|_{S}$ gives a fiberwise hyperplane section of $S$ over $C$ which implies that both $A$ and $B$ are sections of $\pi$. 
Since $A$ intersects $B$ transversally at a single point in $S$ which is contained in the fiber $F_y := \pi^{-1}(y)$, we know that $F_y$ is smooth. Denote by $F_1,\cdots, F_k$ the singular fibers of $\pi$. Then every $F_{i}$ is a union of two different lines where $A$ intersects one line denoted by $\ell_{i}^+$ and $B$ intersects the other line denoted by $\ell_{i}^-$.

Next, we compute some intersection numbers. 
Clearly $\phi^* E = S + \oE$. Thus we have
\[
(A\cdot \oE) + (A\cdot S) = (A\cdot \phi^* E) = (\phi_* A \cdot E) = (C\cdot E) = -1
\]
and  
\[
(A\cdot S) = (A\cdot (A+B))_{\oE} = 1,
\]
which implies that 
\[
(A\cdot (A+B))_{S} = (A\cdot \oE) = -2.
\]
Similarly, 
\[
(B\cdot \oE) + (B\cdot S) = (B\cdot \phi^* E) = (\phi_* B \cdot E) = (C\cdot E) = -1
\]
and  
\[
(B\cdot S) = (B\cdot (A+B))_{\oE} = 1-(m-2),
\]
which implies that 
\[
(B\cdot (A+B))_{S} = (B\cdot \oE) = -2+(m-2) = m-4.
\]
Since $(A\cdot B)_S = 1$, we have that $(A^2)_S = -3$ and $(B^2)_S = m-5$.

Since $\ell_i^{-}$ are disjoint $(-1)$-curves, we can contract all to get a birational morphism $\sigma : S\to S^{\rm m}$ where $\pi$ descends to a smooth $\bP^1$-fibration $\pi^{\rm m}: S^{\rm m}\to C$. Denote by $A^{\rm m}:= \sigma_* A$ and $B^{\rm m}:=\sigma_* B$ the two sections of $\pi^{\rm m}$. Since $\sigma$ is isomorphic in a neighborhood of $A$, we know that 
\[
((A^{\rm m})^2) = (A^2) = -3,\qquad (A^{\rm m}\cdot B^{\rm m}) = 1,\quad\textrm{and}\quad ((B^{\rm m})^2) = (B^2) + k = m-5+k. 
\]
Thus we have $S^{\rm m} \cong \bF_3$ with $A^{\rm}$ the negative section.
Let $F^{\rm m}$ denote the fiber class of $\pi^{\rm m}$. Then $((B^{\rm m} - A^{\rm m})\cdot A^{\rm m}) = 4$ implies that $B^{\rm m} \sim A^{\rm m} + 4F^{\rm m}$. Then we have 
\[
m-5+k = ((B^{\rm m})^2) = (A^{\rm m} + 4F^{\rm m})^2 = 5,
\]
i.e.\ $m = 10 - k$.

When $k\geq 2$, then we have $\hvol(x,X) \leq m \leq 8$ so we are done. When $k  = 1$, then we have $\hvol(x,X) \leq m = 9$. Moreover, the equality cannot hold as otherwise $\ord_E$ is a minimizer of $\hvol_{X,x}$ which by \cite{LX20} implies that $E$ is a Koll\'ar component, contradicting to the fact that $E$ is non-normal.


We are left with the case when $k = 0$, i.e.\ $\pi$ is smooth and $S = S^{\rm m}$. 
By the same argument before, we have  $(B^2)_S = 5$ and $m=10$. Since $\tY$ is smooth, we can run an $(K_{\tY}+\epsilon\oE)$-MMP on $Y$ over $X$ for $0<\epsilon\ll 1$. The Mori cone of curves on $\tY/X$ is generated by $A$ and $F=F^{\rm m}$ as $B\sim A+4F$ on $S$. Moreover, by computations $(A\cdot \oE) = -2$ and $(F\cdot \oE) = 2$, we know that $A$ is the only $(K_{\tY}+\epsilon\oE)$-negative extremal ray. Thus the first step of the MMP gives us a divisorial contraction $\psi:\tY\to Y'$ which contracts $\oE$ to a rational curve $B'$. Since $(A\cdot B)_{\oE} = 1$ we know that $B\to B'$ is bijective. Moreover, by the fact that $(A\cdot (S+\frac{1}{2}\oE) )=0$ we know that $S+\frac{1}{2}\oE = \psi^* S'$ where $S' := \psi_* S$. Since $(\tY, S+\frac{1}{2}\oE)$ is plt by inversion of adjunction, we know that $(Y', S')$ is plt and hence $S'$ is normal. Denote the different divisor of $(Y',S')$ along $S'$ by $\Delta_{S'}$. Since $S\to S'$ only contracts $A$ to a point, we conclude that $S'\cong \bP(1,1,3)$ and hence $-S'|_{S'}$ is ample. This implies that
\[
\hvol(x,X) \leq \hvol_{X,x}(\ord_{S'}) = (-K_{S'}-\Delta_{S'})^2 \leq (-K_{S'})^2 = \tfrac{25}{3}<9.
\]
Thus the proof is finished.
\end{proof}

\begin{prop}\label{prop:crex=2-2}
    Theorem \ref{thm:nv-can} holds if $\crex(x,X) = 2$, the blow-up $\mu:Y\to X$ extracts only one prime divisor $E$, and $E$ is not Cartier in $Y$.
\end{prop}

\begin{proof}
Suppose that $E$ is not Cartier at some point $y\in Y$. Then let $(\ty\in \tY) \to (y\in Y)$ be an index $1$ cover of $E$ at $y$. Then by the same argument as the second paragraph of the proof of Proposition \ref{prop:crex=1}, we know that $Y$ has transversal $A_1$-singularities along a curve $C$ such that $y\in C\subset E$ and $E$ has Cartier index $2$ in $Y$ along $C$. Since $\crex(x, X) = 2$ and both $E$ and the exceptional divisor of the blow-up of $Y$ along $C$ are crepant exceptional divisors over $x\in X$, we know that $C$ is an irreducible smooth projective curve and $(Y_{\sing}\cap E)\setminus C$ is a finite set. In the below, we often treat $E\subset \bP^{e}$ as an embedded projective variety from the blow-up structure. 

Next, let $x\in H$ be a general hyperplane section of $x\in X$. Let $H_Y\subset Y$ be the strict transform of $H$. Then by Lemma \ref{lem:can-elliptic} and Theorem \ref{thm:Laufer-blowup} we know that $x\in H$ is a Gorenstein elliptic surface singularity and $\mu_H: H_Y\to H$ is a canonical modification. Moreover, we know that $\mu_H$ has a single exceptional curve $C_1= E\cap H_Y$ and that $H_Y$ has only $A_1$-singularities at a non-empty finite set $C\cap H_Y$ and smooth elsewhere. Let $f_H: \tH\to H$ be the minimal resolution. Then $f_H$ factors through $g_H:\tH\to H_Y$ being a minimal resolution of the $A_1$-singularities. Let $\tC_2,\cdots, \tC_k$ be the exceptional curves of $g_H$. Denote by $\tC_1\subset \tH$ the strict transform of $C_1$. Then we know that $\mathrm{Ex}(f_H) = \cup_{i=1}^k \tC_i$ and $\tC_2,\cdots, \tC_k$ are disjoint $(-2)$-curves such that $(\tC_1 \cdot \tC_i) = 1$ for any $i\geq 2$. Thus by Theorem  \ref{thm:laufer-resol} we know that $f_H$ is a log resolution, $\tC_1\cong \bP^1$, and $(\tC_1^2)\leq -3$.

Next, we claim that $Z = 2\tC_1 + \sum_{i=2}^k \tC_i$ is the fundamental cycle of $x\in H$ in $\tH$.  Denote by $Z_{\red}:= \sum_{i=1}^k \tC_i$ and $Z':= \tC_1 + Z_{\red}$. Since $E$ is not Cartier in $Y$ we know that $\ord_E \fm_{X,x} \geq 2$ which implies that $Z\geq 2\tC_1$ by \cite[Theorem 4.57]{KM98}. Moreover, $\Supp(Z) = \Supp(Z_{\red})$ implies that $Z \geq Z'$. Then by the definition of Gorenstein elliptic singularities we know that 
\[
0\leq \chi(\cO_{Z'}) = \chi(\cO_{\tC_1}) + \chi(\cO_{Z_{\red}}) - (Z_{\red}\cdot \tC_1) = 2 - (\tC_1^2) - (k-1). 
\]
Here we use the fact that $Z_{\red}$ is a tree of $\bP^1$'s as otherwise we have $\chi(\cO_{Z_{\red}})\leq 0$ which contradicts $x\in H$ being Gorenstein elliptic. 
Thus we have 
$(Z'\cdot \tC_i) = 0$ for any $i\geq 2$ and 
\[
(Z'\cdot \tC_1) = 2(\tC_1^2) + (k-1) \leq (\tC_1^2) + 2 <0. 
\]
As a result, we have $Z\leq Z'$ which implies $Z = Z'$ combining with the earlier discussion. In particular, $\chi(\cO_Z)=0$ implies that $(\tC_1^2) = 3-k$. 

Next, we show that $C_1$ is a rational normal curve. Indeed, we have a short exact sequence
\[
0\to \cO_{\tH}(-Z_{\red})/\cO_{\tH}(-Z) \to \cO_Z \to \cO_{Z_{\red}} \to 0.
\]
Since $Z = Z_{\red}+ \tC_1$, we know that  $\cO_{\tH}(-Z_{\red})/\cO_{\tH}(-Z) \cong \cO_{\tC_1}\otimes \cO_{\tH} (-Z_{\red})$. 
Thus twisting the exact sequence by $\cO_{\tH}(-Z)$ and taking long exact sequence yields
\[
H^0(Z, \cO_Z(-Z)) \to H^0(Z_{\red}, \cO_{Z_{\red}}(-Z)) \to H^1(\tC_1, \cO_{\tC_1}\otimes \cO_{\tH} (-Z_{\red}-Z)).
\]
Since $(\tC_1^2)=3-k \leq -3$, we have $k\geq 6$. Thus we have
\begin{align*}
\deg \cO_{\tC_1}\otimes \cO_{\tH} (-Z_{\red}-Z) & = (\tC_1 \cdot (-Z_{\red}- Z) = -3(\tC_1^2) - 2(k-1)\\ & = -3(3-k) - 2(k-1) = k-7\geq -1.
\end{align*}
Thus $H^1(\tC_1, \cO_{\tC_1}\otimes \cO_{\tH} (-Z_{\red}-Z))=0$ as $\tC_1\cong \bP^1$. Hence we have a surjection $H^0(Z, \cO_Z(-Z)) \twoheadrightarrow H^0(Z_{\red}, \cO_{Z_{\red}}(-Z))$. Since $Z_{\red}$ is a tree of $\bP^1$'s where $\cO_{Z_{\red}}(-Z)$ is trivial on all the $(-2)$-curves $\tC_i$ for $i\geq 2$, we know that $H^0(Z_{\red}, \cO_{Z_{\red}}(-Z)) \cong H^0(\tC_1, \cO_{\tC_1}(-Z))$ which implies that $\tC_1\to C_1\subset \bP^{e-1}$ is defined by a complete linear system. Thus $C_1\subset \bP^{e-1}$ is a rational normal curve in some projective subspace. Since $C_1$ is a general hyperplane section of $E\subset \bP^e$, we know that  $E$ is a surface of minimal degree in some projective subspace. 

Next, we use the classification of surfaces of minimal degree \cite{EH87} to conclude that $E$ is $\bP^2$ embedded via $|\cO(1)|$ or $|\cO(2)|$, a quadric surface in $\bP^3$, a cone over a rational normal curve $\bF_{a;0}$, or a smooth rational normal scroll $\bF_{a;b}$ with $b\geq 1$. In all cases, we have $(E, \frac{1}{2}C)$ is klt and $-E$ is ample which implies that $E$ is a Koll\'ar component over $x\in X$. Here we use the fact that $C$ is a smooth Cartier divisor in $E$ along which $E$  has index $2$ in $Y$. Thus we have
\[
\hvol(x,X) \leq \hvol_{X,x}(\ord_E) = (-K_E - \tfrac{1}{2}C)^2 < (-K_E)^2.
\]
If $E$ is $\bP^2$ or a quadric surface, then $\hvol(x,X) <(-K_E)^2 \leq 9$.
If $E\cong\bF_{a;0}\cong \bP(1,1,a)$, then by embedding dimension constraints as $E\subset Y$ and $Y$ has only hypersurface singularities we know that $a\leq 3$. Hence 
$\hvol(x,X) < (-K_E)^2 = \frac{(a+2)^2}{a}\leq 9$.

The only case left is when $E\cong \bF_{a;b}$ is a smooth Hirzebruch surface. Since $\dim X = 3$, we know that $\mu:(Y,E) \to X$ is a plt blow-up.
Let $L_E$ be the nef line bundle on $E$ as the pullback of $\cO_{\bP^1}(1)$ under the $\bP^1$-fibration $E\to \bP^1$. Then by Proposition \ref{prop:Pic-lc} after an \'etale base change $(x'\in X') \to (x\in X)$ we can contract $E'= E\times_X X'$ in $Y'= Y\times_X X'$ along the fibers to get a factorization $Y'\to W'\to X'$ such that the exceptional locus of $W'/X'$ is a rational curve $\Gamma'$. Moreover, since $Y'$ has transversal $A_1$-singularities along $C'= C\times_X X'$, there are two possibilities: either $C'$ is contracted to a point $w'\in \Gamma'$ which creates a singularity with $\crex(w',W')=1$, or $C'$ dominates $\Gamma'$ which implies that $W'$ has a transversal $A_k ~(k\geq 2)$, $D_k$, or $E_k$-singularity at a general point $w\in \Gamma'$. In the former possibility, by Theorem \ref{thm:nv-crepant} and Proposition \ref{prop:crex=1} we have 
\[
\hvol(x,X)= \hvol(x', X') < \hvol(w',W') \leq 9.
\]
In the latter possibility, by Theorem \ref{thm:nv-crepant} we have
\[
\hvol(x,X)= \hvol(x', X') < \hvol(w,W') \leq 9.
\]
Thus the proof is finished.
\end{proof}


\begin{prop} \label{prop:crex=2-3}
    Theorem \ref{thm:nv-can} holds if $\crex(x,X) = 2$ and the blow-up $\mu:Y\to X$ extracts precisely two prime divisors $E_1$ and $E_2$.
\end{prop}

\begin{proof}
Denote by $E:=E_1+E_2$. 
We know $Y$ has cDV singularities from earlier discussions. Moreover, the set $Y_{\sing}\cap E$ is finite by the assumption $\crex(x,X) = 2$.  Let $x\in H$ be a general hyperplane section, which is a Gorenstein elliptic surface singularity. Then by Bertini's theorem and Theorem \ref{thm:Laufer-blowup} we know that $\Bl_x H\to H$ is a minimal resolution with precisely two exceptional curves. Then by Proposition \ref{prop:fund-cycle-2} we have two $\bP^1$'s intersecting transversally at two points or forming a tacnode, and the fundamental cycle of $x\in H$ in $\Bl_x H$ is reduced. This implies that $\cO_Y(-E)$ is the inverse image of $\fm_{X,x}$ and hence $E$ is Cartier in $Y$. By adjunction, we know that $E$ is a reduced Gorenstein del Pezzo surface, each irreducible component $E_i$ is normal, and at every generic point of $C:=E_1\cap E_2$ the intersection is either transversal or tacnodal. Moreover, since $\mu$ is crepant and $-E$ is Cartier and ample over $X$, whenever $(Y,E)$ is lc below (or equivalently by inversion of adjunction, $E$ is slc), we have that $\mu:(Y,E)\to X$ is an lc blow-up.

Since $A_X(E_i)= 1$ for $i = 1,2$, we have $\hvol(x,X)\leq \vol_{X,x}(\ord_{E_i})$. Moreover, we know that $\fa_j(\ord_{E_i}) \supset \fa_j(\ord_{E_1}) \cap \fa_j(\ord_{E_2}) = \mu_*\cO_Y(-jE)$ which implies that $\vol_{X,x}(\ord_{E_i})\leq -(-E)^3 = (-K_E)^2$. Combining these inequalities, we get $
\hvol(x,X) \leq (-K_E)^2$.
Moreover, if equality holds, then both $\ord_{E_1}$ and $\ord_{E_2}$ are minimizers of $\hvol_{X,x}$, contradicting the uniqueness of the minimizer (Theorem \ref{thm:SDC}). Thus we have 
\[
\hvol(x, X) < (-K_E)^2.
\]

Below, we split into three cases based on the classification from \cite[Section 1.3]{Rei94}. Note that $(-K_E)^2$ is the sum of degrees of $E_1$ and $E_2$ viewed as subvarieties of $\bP^e$. 

\smallskip

\textbf{Case 1.} \emph{($=$Case (B) from \cite[Section 1.3]{Rei94})} Two surfaces from (a1),(b),(c0),(d0),(e), not both from (a1).\footnote{Here and later in the proof, we follow the notation of \cite[Table in Theorem 1.1]{Rei94} for surfaces from (a1),(a2),(a3),(b),(c1),(c2),(c0),(d1),(d0),(e).} 

\emph{Subcase 1.1.} If none of the components is (e), then we get $(-K_E)^2 \leq 8$ and we are done. 

\emph{Subcase 1.2.} If precisely one of the components (say $E_2$) is (e), then there is a nef line bundle $L_E$ on $E$ such that $L_E|_{E_1}\cong -K_E|_{E_1}$ is ample and $L_E|_{E_2}$ is the pull-back of $\cO_{\bP^1}(2)$ under the $\bP^1$-fibration $E_2\to \bP^1$. Moreover, we know that $E$ is slc and Cartier in $Y$ which implies that $\mu:(Y,E)\to X$ is an lc blow-up. Thus by Proposition \ref{prop:Pic-lc}, after an \'etale base change $(x'\in X')\to (x\in X)$ we can contract $E_2' = E_2\times_X X'$ in $Y'=Y\times_X X'$ along the $\bP^1$-fibers only to get a factorization $Y'\xrightarrow{g'} Z'\to X'$. Let $C_{Z'}:=g'(E_2')\subset Z'$ which is a rational curve. Denote by $E_{Z'}:=g'_* E_1'= g'_* (E_1\times_X X')$ in $Z'$. Then we know that $Z'$ has transversal $A_1$-singularities at a general point of $C_{Z'}$ and that $g'^* E_{Z'} = E_1' + \frac{1}{2}E_2'$. Thus by adjunction we know that $(Z', E_{Z'})$ is plt and $\frac{1}{2}C_{Z'}$ is the different divisor on $E_{Z'}$. In particular, $E_{Z'}$ is normal and hence $E_1\cong E_1'\cong E_{Z'}$ by Zariski's main theorem. Thus we have $(E_{Z'},\frac{1}{2}C_{Z'}) \cong (E_1, \frac{1}{2}C)$ is a klt log Fano pair by running through all possibilities of (a1), (b), (c0), and (d0) where $E_1$ is one of $\bP^2$, $\bF_{2;0}$, $\bF_{0;1}$, or $\bF_{1;1}$. In particular, $Z'\to X'$ is a crepant plt blow-up extracting a Koll\'ar component $E_{Z'}$. Thus we have 
\[
\hvol(x,X) = \hvol(x', X') \leq (-K_{E_{Z'}} - \tfrac{1}{2}C_{Z'})^2 = (-K_{E_1} - \tfrac{1}{2}C)^2 < (-K_{E_1})^2 \leq 9.
\]
Here we use the fact that $-K_{E_1}$ is ample.


\emph{Subcase 1.3.} If both components are (e), then we have a fibration $\pi: E_1\cup E_2 \to \bP^1$ where each fiber of $\pi$ is a union of two $\bP^1$'s at a point. Thus by choosing $L_E:=\pi^*\cO_{\bP^1}(1)$ and applying Proposition \ref{prop:Pic-lc}, after an \'etale base change $(x'\in X')\to (x\in X)$ we get a factorization $Y'\xrightarrow{g'} Z'\to X'$ such that $g'$ contracts $E_1'\cup E_2'$ along the fibers of $\pi$ to a curve $\Gamma'\subset Z'$. Thus $Z'$ has a transversal $A_2$-singularity at a general point $z\in \Gamma'$. By Theorem \ref{thm:nv-crepant} we have 
\[
\hvol(x,X) = \hvol(x', X') < \hvol(z, Z') = 9.
\]


\smallskip

\textbf{Case 2.} \emph{($=$Case (D$_2$) from \cite[Section 1.3]{Rei94})} 
Two surfaces from (a3),(c2), not both from (a3). 

By embedding dimension constraints, if we take a surface $\bF_{a;0}$ from (c2) then its degree $a\leq 3$, which implies that $(-K_E)^2\leq 6$ (note that (a3) has degree $1$).

\smallskip

\textbf{Case 3.} \emph{($=$Case (C$_2$) from \cite[Section 1.3]{Rei94})}  Two surfaces from (a2),(c1),(d1), not both from (a2). Then we know that $\nu: E_1\sqcup E_2 = \oE \to E$ is the normalization, and that $\mu:(Y,E)\to X$ is an lc blow-up.

\emph{Subcase 3.1.} If none of the components is (d1), then similar to Case 2 we have the sum of degrees $(-K_E)^2 \leq 6$. 

\emph{Subcase 3.2.} If both components are (d1) (say $E_i\cong \bF_{a_i;1}$), then we analyze the gluing of $E_1$ and $E_2$. Let $A_i$ and $B_i$ denote the fiber and negative section in $E_i$ respectively such that $A_i+B_i$ is the conductor divisor in $E_i$. There are two situations in the gluing process: either $\nu(A_i) = \nu(A_{3-i})$ and $\nu(B_i) = \nu(B_{3-i})$, or $\nu(A_i) = \nu(B_{3-i})$. 

In the former situation, the two $\bP^1$-fibrations $E_i\to \bP^1$ glue to a fibration $E\to \bP^1$ with a general fiber being a union of two $\bP^1$'s at a point. Denote by $L_E$ the pull-back of $\cO_{\bP^1}(1)$ to $E$. Then applying Proposition \ref{prop:Pic-lc} as before, after an \'etale base change $(x'\in X') \to (x\in X)$ we have a factorization $Y'=Y\times_X X'\xrightarrow{g'} Z'\to X'$  such that $g'$ contracts $E$ to a rational curve $\Gamma'$ over $x\in X$. Moreover, $Z'$ has transversal $A_2$-singularities at a general point $z'\in \Gamma'$. Thus by Theorem \ref{thm:nv-crepant} we have 
\[
\hvol(x,X) = \hvol(x', X') < \hvol(z', Z')  = 9.
\]

In the latter situation, denote by $C_i:= \nu(B_i)$ for $i=1,2$ and $y:= C_1\cap C_2$. We will follow similar arguments to Case 2 of the proof of Proposition \ref{prop:crex=1}. Then we have exact sequences
\[
0 \to \cT_{C_i} \to \cT_{\oE}|_{A_{3-i}} \oplus \cT_{\oE}|_{B_{i}} \to \cE_i\to 0,
\]
where $\cE_i$ is a subsheaf of $\cT_Y|_{C_i}$ whose quotient is torsion. For a fixed $i\in \{1,2\}$, we may choose   suitable local analytic coordinates $(x_1,\cdots, x_4)$ at $y$ such that $E = V(x_3(x_1x_2 -x_3), x_4)$ and $C_i = V(x_2, x_3, x_4)$. Write $Y = V(f)$ where $ f= x_3(x_1x_2 -x_3) h + x_4 g$ such that $x_4$ does not appear in $h$. Since $C_i\not\subset Y_{\sing}$ we know that $g|_{C_i} \neq 0$ which gives $l_i:=\ord_y (g|_{C_i}) \in \bZ_{\geq 0}$. Then locally at $y$ we have $\cE_i = \langle \partial_{x_1}, \partial_{x_2}, x_1\partial_{x_3} \rangle$ which implies that a local generator of $\det \cE_i$ at $y$ is given by $\partial_{x_1}\wedge\partial_{x_2}\wedge x_1\partial_{x_3} = x_1 \partial_{x_1}\wedge\partial_{x_2}\wedge \partial_{x_3}$. Thus the natural pairing of $\det\cE_i$ and $\omega_Y|_{C_i}$ leads to a duality between $(\det\cE_i)\otimes \cO_{C_i}((1-l_i) y)$ and $\omega_Y|_{C_i}$ at $y$. Similarly, at each point $y_{i,j}'\in Y_{\sing}\cap (C_i \setminus \{y\})$ we have $l_{i,j}'\in \bZ_{>0}$ such that  $(\det\cE_i)\otimes \cO_{C_i}(-l_{i,j}' y_{i,j}')$ and $\omega_{Y}|_{C_i}$ are dual to each other at $y_{i,j}'$. Combining these together, we get $(\det \cE_i) \otimes \cO_{C_i}((1-l_i)y - \sum_j l_{i,j}' y_{i,j'}) \cong \omega_Y^{\vee}|_{C_i}$  which implies that 
\begin{align*}
 0 & = \deg \omega_Y^{\vee}|_{C_i} = \deg (\det\cE_i) \otimes \cO_{C_i} ( (1-l_i) y - \sum_{j} l_{i,j}' y_{i,j}')\\
 & = \deg \cE_i + 1-l_i - \sum_{j} l_{i,j}' \\
 &\leq \deg \cE_i + 1 = -\deg \cT_{C_i} +\deg \cT_{\oE}|_{A_{3-i}} + \deg \cT_{\oE}|_{B_{i}} + 1\\
 & = \deg \cT_{C_i} + \deg \cN_{A_{3-i}/E_{3-i}} + \deg \cN_{B_{i}/E_{i}} + 1\\
 & = 2 - a_{i} + 1  = 3-a_{i}.
\end{align*}
Thus we get $a_i\leq 3$ which implies that $(-K_E)^2 = a_1+a_2 + 4\leq 9$ unless $a_1 = a_2 = 3$ which we treat separately below.

If we have $a_1 = a_2 = 3$, then the above inequalities are equalities which implies that $l_i$ and $l_{i,j}'$ must all vanish. Thus $Y$ is smooth along $C_1$ and $C_2$ and hence smooth along $E$. Let $L:= -2E_1-E_2$ be a Cartier divisor on $Y$. Since $E_1\cong E_2 \cong \bF_3$, simple computations show that 
\[
(L\cdot A_1) = (L\cdot B_2) = 3, \qquad (L\cdot A_2) = (L\cdot B_1) = 0.
\]
In particular, $L$ is nef over $X$ such that $L|_{E_1}$ is big and $L|_{E_2}$ is not big. By the basepoint free theorem, $L$ is semiample over $X$ whose ample model gives a divisorial contraction $g: Y\to Z$ that contracts $E_2$ along the fibers and contract the negative section of $E_1$ to a point. Denote by $E_Z:= g_* E_1$. Then we know that $-2E_Z = g_*L$ which implies that $E_Z$ is $\bQ$-Cartier and $g^*E_Z = E_1+\frac{1}{2}E_2$. 
Since $E_2|_{E_1} = C_1+ C_2$ is snc, we know that $(E_1, \frac{1}{2}E_2|_{E_1})$ is klt. Thus by inversion of adjunction we know that $(Y, E_1+\frac{1}{2}E_2)$ is plt, and so is $(Z, E_Z)$ as $g$ is crepant. Moreover, $-E_Z=\frac{1}{2}g_*L$ is ample. Thus we know that $Z\to X$ is a plt blow-up extracting a Koll\'ar component $(E_Z, \Delta_{E_Z})$. Since $E_Z$ is normal, we have $E_Z\cong \bP(1,1,3)$. Hence we have 
\[
\hvol(x,X) \leq \hvol_{X,x}(\ord_{E_Z}) = (-K_{E_Z}-\Delta_{E_Z})^2 \leq (-K_{E_Z})^2 = \tfrac{25}{3}<9.
\]


\emph{Subcase 3.3.} If one of the components is (d1) (say $E_2$) and the other is (a2) (say $E_1$), then similar argument to Subcase 3.2 shows that if $C_1\subset E$ is the gluing of a line $\ell$ in $E_1=\bP^2$ and a negative curve $B$ in $E_2=\bF_{a;1}$ then $\deg \cN_{\ell/\bP^2} = 1$ and $\deg \cN_{B/E_2}=-a$, and hence we have an inequality
\[
0 \leq \deg \cT_{C_1} + \deg \cN_{\ell/\bP^2} + \deg \cN_{B/E_2} + 1 = 2 + 1 - a + 1 = 4 - a,
\]
which yields $a\leq 4$. Thus $(-K_E)^2 = 1+ a + 2 \leq 7$. 

\emph{Subcase 3.4.}
The only subcase left is when the components are (d1) (say $E_2\cong \bF_{a_2;1}$) and (c2) (say $E_1\cong \bF_{a_1;0}$). By embedding dimension constraints we have $a_1=2$ or $3$. 
Let $A$ and $B$ denote the fiber and negative section in $E_2$ such that $A+B$ is the conductor divisor in $E_2$. Let $\ell$ and $\ell'$ denote the lines in $E_1$ such that $\ell+\ell'$ is the conductor divisor in $E_1$,  $\nu(\ell) = \nu(B)=: C_1$ and $\nu(\ell')= \nu(A)=: C_2$. Denote by $y:= C_1\cap C_2$. 

We choose a local analytic coordinate $(x_1,\cdots, x_4)$ at $y$ such that $C_1 = V(x_2,x_3,x_4)$, $C_2 = V(x_1, x_2, x_4)$, $E_2 = V(x_2,x_4)$, and 
\[
E_1 = \begin{cases} V(x_2^2 - x_1x_3, x_4) & \textrm{ if }a_1=2;\\
V(x_2^2 - x_1x_4, x_4^2 - x_2 x_3, x_1x_3 - x_2 x_4) & \textrm{ if }a_1 = 3.
\end{cases}
\]
This is possible as $E=E_1+E_2$ is slc where $(E_2, A+B)$ is snc, hence the complete local ring $\widehat{\cO_{E,y}}$ only depends on the choice of $a_1$ by Koll\'ar's gluing theory. 
Let $\sigma: \tE_1\to E_1$ be the minimal resolution with $\tilde{\ell}\subset \tE_1$ the strict transform of $\ell$. Let $\ty:= \sigma^{-1}(y) \cap \tilde{\ell}$. In particular, we have $\tE_1\cong \bF_{a_1}$ and $\ty$ is the intersection of a fiber $\tilde{\ell}$ with the negative section $\sigma^{-1}(y)$. Thus we can choose local analytic coordinates $(u,v)$ of $\tE_1$ and write
\[
\sigma(u,v) = \begin{cases}
    (u, uv, uv^2, 0)  &  \textrm{ if }a_1 = 2;\\
    (u, uv, uv^3, uv^2) & \textrm{ if }a_1 = 3.
\end{cases}
\]
Here $C_1=\nu(\ell)$ is given by the image of $\tilde{\ell} = (v=0)$. Moreover, we know that the tangent space $\cT_{E_2, y} = V(x_2, x_4)$. 

Next, let $\cE$ be the image of $\sigma_* \cT_{\tE_1}|_{\tilde{\ell}} \oplus \cT_{E_2}|_{B} \to \cT_Y|_{C_1}$. By the expression of $\sigma$ we know that the image of  $\sigma_* \cT_{\tE_1}|_{\tilde{\ell}} $ is $\langle \partial_{x_1}, x_1 \partial_{x_2}\rangle$, and the image of  $\cT_{E_2}|_{B}$ is $\langle \partial_{x_1}, \partial_{x_3}\rangle$, which implies that $\cE = \langle\partial_{x_1}, x_1 \partial_{x_2}, \partial_{x_3}\rangle$. Thus we still have a short exact sequence
\[
0 \to \cT_{C_1}\to \sigma_* \cT_{\tE_1}|_{\tilde{\ell}} \oplus \cT_{E_2}|_{B} \to \cE \to 0.
\]
Thus a local generator of $\det \cE$ at $y$ is $\partial_{x_1}\wedge x_1 \partial_{x_2}\wedge \partial_{x_3} = x_1 \partial_{x_1}\wedge \partial_{x_2}\wedge \partial_{x_3}$. Now, let $Y = V(f)$. Since $Y$ is smooth at a general point $y'$ of $C_1$, we know that the tangent space $\cT_{Y, y'}$ contains $\cE_{y'}$ which implies that $\cT_{Y,y'} = \langle\partial_{x_1}, \partial_{x_2}, \partial_{x_3}\rangle$. In particular, we know that $f_{x_4}(y') \neq 0$ which implies that $f_{x_4}|_{C_1}\neq 0$ and hence $l : = \ord_y (f_{x_4}|_{C_1})\in \bZ_{\geq 0}$. Hence a local generator of $\omega_Y|_{C_1}$ at $y$ is given by $x_1^{-l} dx_1\wedge dx_2 \wedge dx_3$. In particular, we have a duality between $(\det\cE)\otimes \cO_{C_1}((1-l)y)$ and $\omega_Y|_{C_1}$ at $y$. Then by similar arguments as in Subcase 3.2, we have
\[
0 \leq \deg \cT_{C_1} + \deg \cN_{\tilde{\ell}/\tE_1} + \deg \cN_{B/E_2} + 1  = 2 + 0 - a_2 + 1,
\]
which yields $a_2 \leq 3$. 
Thus we get  $(-K_E)^2 = a_1+a_2+2 \leq 8$. This completes the proof.
\end{proof}

\begin{proof}[Proof of Proposition \ref{prop:crex=2}]
This follows directly from combining Propositions \ref{prop:crex=2-1}, \ref{prop:crex=2-2} and \ref{prop:crex=2-3}.
\end{proof}

\subsection{General cases}

\begin{proof}[Proof of Theorem \ref{thm:nv-can}]
We proceed by induction on $\crex(x,X)$. If $\crex(x,X) \leq 2$, then the result follows from Propositions \ref{prop:crex=1} and \ref{prop:crex=2}. Now assume that Theorem \ref{thm:nv-can} is true for $\crex(x,X) \leq d-1$ with $d \geq 3$. Assume $\crex(x,X) = d$ from now on.

Firstly, we reduce to the case when $x\in X$ is $\bQ$-factorial. If not, let $\hX\to X$ be a small $\bQ$-factorialization that is isomorphic away from $x$. 
Then $\hX$ has cDV singularities, as otherwise we can find $\hx\in \hX$ over $x$ such that $0<\crex(\hx,\hX)\leq \crex(x,X) = d$ and the new singularity  $\hx\in \hX$ is $\bQ$-factorial and $\hvol(x,X) <\hvol(\hx, \hX)$ by  Theorem \ref{thm:nv-crepant}. Let $\cup_{i=1}^l C_i$ be the exceptional curves of $\hX \to X$. Then $\hX_{\sing}$ will contain some $C_i$ as otherwise $\crex(x,X) = 0$. If $\hX$ has transversal $A_{\geq 2}$, $D_k,$ or $ E_k$-singularity along some $C_i$, then by Theorem \ref{thm:nv-crepant} we have 
\[
\hvol(x, X) < \hvol(\hx, \hX) \leq 9
\]
for a general point $\hx\in C_i$.
If $\hX$ has transversal $A_1$-singularity along some $C_i$, then by Proposition \ref{prop:Pic-flop} after an \'etale base change $(x'\in X')\to (x\in X)$ we can contract $C_i'=C_i\times_X X'$ only in $\hX':=\hX\times_X X'$, hence creating a new singularity $z\in Z'$ with $\crex(z, Z') = 1$. By Proposition \ref{prop:crex-analytic} we have $\crex(x', X') = \crex(x,X) >1$ which implies that $z$ is in the exceptional locus of the crepant birational morphism $Z'\to X'$. 
Thus by Theorem \ref{thm:nv-crepant} and Proposition \ref{prop:crex=1} we have
\[
\hvol(x, X) = \hvol(x', X') < \hvol(z, Z') \leq 9. 
\]

Next, assume $x\in X$ is $\bQ$-factorial. Let $Y\to X$ be a maximal crepant model as in \cite[Lemma 3.2]{LX19}. Let $E_1\subset Y$ be a crepant exceptional prime divisor over $x\in X$. Then running $(K_Y+\epsilon E_1)$-MMP on $Y/X$ for $0<\epsilon\ll 1$ we get $Y\dashrightarrow Y^+ \to Y'$ over $X$ where $Y\dashrightarrow Y^+$ is a composition of flips and $Y^+\to Y'$ only contracts the birational transform of $E_1$. Since $Y^+$ is also a maximal crepant model over $X$, we may replace $Y$ by $Y^+$. 

If $E_1$ is contracted to a point $y'\in Y'$, then we have $\crex(y', Y')=1$. Thus by Theorem \ref{thm:nv-crepant} and Proposition \ref{prop:crex=1} we have
\begin{equation}\label{eq:nv-compare}
\hvol(x,X) < \hvol(y', Y') \leq 9. 
\end{equation}

If $E_1$ is contracted to a curve $\Gamma\subset Y'$, then we know that $Y'\to X$ has exceptional locus of pure codimension $1$ by \cite[Corollary 2.63]{KM98} which implies that $\Gamma$ is contained in some crepant exceptional divisor $E_2$ in $Y'$ over $X$. Then we run $(K_{Y'} + \epsilon E_2)$-MMP to get $Y'\dashrightarrow {Y'}^+ \to Y''$ where $Y'\dashrightarrow {Y'}^+$ is a composition of flips and ${Y'}^+\to Y''$ only contracts the birational transform of $E_2$. In each step of the flips from $Y'\dashrightarrow {Y'}^+$, we cannot contract $\Gamma$ as otherwise we create a singularity $z\in Z$ with $\crex(z,Z) = 1$ which implies $\hvol(x,X)<9$ by the same argument as in \eqref{eq:nv-compare}. Thus $\Gamma$ survives as a curve $\Gamma^+$ in ${Y'}^+$ which is contained in $E_2^+$. If the image of $\Gamma^+$ in $Y''$ is a point $y''$ while $E_2^+$ gets mapped to a curve, then we get $\crex(y'', Y'') = 1$ and we are done by the same argument as in \eqref{eq:nv-compare}. If the image of $\Gamma^+$ in $Y''$ is a curve $\Gamma''$, then we know that $E_2^+$ is also mapped onto $\Gamma''$. Thus there are two crepant exceptional divisors over the generic point of $\Gamma''$. As a result, for a general point $y''\in \Gamma''$ we know that $y''\in Y''$ is a transversal $A_{\geq 2}$, $D_k$, or $E_k$-singularity, which implies that 
\[
\hvol(x, X) < \hvol(y'', Y'')\leq 9.
\]
The last possibility is that $E_2^+$ gets contracted to a point $y''\in Y''$. Then we have created a singularity with $\crex(y'', Y'') =2$. Thus by Theorem \ref{thm:nv-crepant} and  Proposition \ref{prop:crex=2} we have 
\[
\hvol(x, X) < \hvol(y'', Y'')\leq 9.
\]
The proof is finished.
\end{proof}

\begin{proof}[Proof of Theorem \ref{thm:nv-3fold}]
    The ``if'' part is a direct consequence of Example \ref{expl:quotient} and Proposition \ref{prop:nv-hyp}.
    For the ``only if'' part, let $(\tx\in \tX)\to (x\in X)$ be the index $1$ cover of $K_X$ whose Cartier index is denoted by $d$. Then by Theorem \ref{thm:finite-deg} we have 
    \[
    \hvol(\tx, \tX) = d\cdot \hvol(x,X)\geq 9d.
    \]
    If $d=1$, i.e.\ $x\in X$ is Gorenstein, then by Theorem \ref{thm:nv-can} we know that  $x\in X$ is either a hypersurface singularity or a cyclic quotient singularity of type $\frac{1}{3}(1,1,1)$. The hypersurface singularity case follows from Proposition \ref{prop:nv-hyp}. If $d\geq 2$, then we have $\hvol(\tx, \tX) \geq 18$ which implies that $\tx\in \tX$ is smooth and $d \leq 3$ by Theorem \ref{thm:ODP-gap}. Thus $x\in X$ is a cyclic quotient singularity of order $2$ or $3$. If $d=2$, then $x\in X$ is of type $\frac{1}{2}(1,1,1)$ as the other case $\frac{1}{2}(1,1,0)$ is Gorenstein. If $d=3$, then $x\in X$ is of type $\frac{1}{3}(1,1,0)$ or $\frac{1}{3}(1,1,2)$ as the other two cases $\frac{1}{3}(1,2,0)$ or $\frac{1}{3}(1,1,1)$ are both Gorenstein.
\end{proof}

\section{Applications and examples}

In this section, we prove Theorems \ref{thm:K-moduli} and \ref{thm:nv-mld}. We give examples of singularities whose local volumes are at least $9$. In the end, we propose some natural questions.

\begin{proof}[Proof of Theorem \ref{thm:K-moduli}]
 Let $x\in X$ be a non-smooth point.
 By \cite{Liu18} we know that $\hvol(x,X) \geq \frac{27}{64}V$.

 (1) Since $V\geq 26$, we have
 \[
 \hvol(x,X) \geq \tfrac{351}{32}\approx 10.97 > \tfrac{32}{3}.
 \]
 Thus by Theorem \ref{thm:nv-3fold} and Proposition \ref{prop:nv-cA2} we know that $x\in X$ is a hypersurface singularity of type $cA_1$ or a cyclic quotient singularity of type $\frac{1}{2}(1,1,1)$. 


 (2) Since $V\geq 22$, we have
 \[
  \hvol(x,X) \geq \tfrac{297}{32}\approx 9.28 > 9.
 \]
 By Theorem \ref{thm:nv-3fold} and Proposition \ref{prop:non-iso-hyp}, we know that $x\in X$ is either an isolated hypersurface singularity of type $cA_1$ or $cA_2$, a transversal $A_1$-singularity, a $D_\infty$-singularity, or a cyclic quotient singularity of type $\frac{1}{2}(1,1,1)$. 

 (3) Let $x\in X$ be a non-Gorenstein point. Since $V\geq 11$, we have
 \[
  \hvol(x,X) \geq \tfrac{297}{64}\approx 4.64 > \tfrac{9}{2}.
 \]
 Denote by $(\tx\in \tX)\to (x\in X)$ the index $1$ cover of $K_X$. Then by Theorem \ref{thm:finite-deg} we have 
 \[
 \hvol(\tx,\tX)\geq 2 \hvol(x,X) > 9. 
 \]
 Thus by Theorems \ref{thm:nv-can} and \ref{thm:nv-3fold} we know that $\tx\in \tX$ is a hypersurface singularity of type $cA_{\leq 2}$. 
 
 Finally, we note that cyclic quotient singularities of type $\frac{1}{2}(1,1,1)$ are not $\bQ$-Gorenstein smoothable by \cite{Sch71}. 
 Thus the proof is finished.
\end{proof}

\begin{rem}\label{rem:K-mod}
According to the author's knowledge based on \cite{SS17, LX19, ADL21, ACC+, ACD+, CDG+, CT23, DJKHQ24, LZ24, Zha24, CFFK24}, we give a list of families of smooth Fano threefolds with volume $V\geq 11$ whose K-moduli compactifications are currently unknown. We hope Theorem \ref{thm:K-moduli} can be used to study these families in the future. We follow Mori--Mukai's notation for the numbering of these families.
\begin{enumerate}
    \item $V\geq 26$: \textnumero 2.20, 2.21, 3.11;
    \item $22\leq V<26$: \textnumero 1.10, 2.16, 2.17, 3.6, 3.7, 3.8;
    \item $11\leq V<22$: \textnumero 1.6, 1.7, 1.8, 1.9, 2.5, 2.6, 2.7, 2.8, 2.9, 2.10, 2.11, 2.12, 2.13, 2.14, 3.1, 3.2, 3.3, 3.4, 3.5.
\end{enumerate}
\end{rem}

\begin{proof}[Proof of Theorem \ref{thm:nv-mld}]
Let $(\tx\in \tX)\to (x\in X)$ be an index $1$ cover of $K_X$. Denote by $d$ the Cartier index of $K_X$ at $x$. 
Then by \cite[Proposition 5.20]{KM98} and Theorem \ref{thm:finite-deg} we have
\[
\mld(\tx, \tX) \leq d\cdot \mld(x, X), \qquad \hvol(\tx, \tX) = d\cdot \hvol(x,X). 
\]
Thus the inequality \eqref{eq:nv-mld} reduces to showing 
\begin{equation}\label{eq:nv-mld-cover}
    \hvol(\tx,\tX) \leq 9 \cdot\mld(\tx,\tX)
\end{equation} 
where $\tx\in \tX$ is Gorenstein canonical. 

Since $K_{\tX}$ is Cartier, we know that $\mld(\tx,\tX) \geq 1$ is a positive integer. If $\crex(\tx, \tX) > 0$, then $\mld(\tx, \tX) = 1$ and \eqref{eq:nv-mld-cover} follows by Theorem \ref{thm:nv-can} with equality if and only if $\tx\in \tX$ is a cyclic quotient singularity of type $\frac{1}{3}(1,1,1)$. If $\crex(\tx,\tX) =0$, then by Theorem \ref{thm:crex-cDV} we know that $\tx\in \tX$ is a cDV singularity. Then we have three cases: $\tx\in \tX$ is smooth, non-smooth  terminal, or non-terminal. The smooth case is obvious as $\mld(\tx,\tX) = 3$ and $\hvol(\tx,\tX) = 27$. For the terminal case, we have $\mld(\tx,\tX) = 2$ by \cite{Mar96}. Thus \eqref{eq:nv-mld-cover} follows from Theorem \ref{thm:ODP-gap} as $\hvol(\tx, \tX)\leq 16 < 18 = 9\cdot \mld(\tx,\tX)$. The last case is when $\tx\in \tX$ is not terminal. Let $\phi: \tZ\to \tX$ be a terminalization. Then we know that $\phi^{-1}(\tx)$ has dimension $1$ since $\phi$ is exceptional over $\tx$ and there are no crepant exceptional divisors over $\tx\in \tX$ by the assumption that $\crex(\tx, \tX) = 0$. Let $\eta$ be a generic point of $\phi^{-1}(\tx)$. Since $\tZ$ is terminal, we know that $\tZ$ is smooth at $\eta$.  Thus we have $\mld(\tx,\tX) \leq \mld(\eta, \tZ) = 2$. Moreover, since $\crex(\tx,\tX) = 0$ we have $\mld(\tx,\tX)>1$. Combining these together, we have $\mld(\tx, \tX) = 2$. Thus \eqref{eq:nv-mld-cover} follows again from Theorem \ref{thm:ODP-gap}.

Finally, we analyze the equality cases of \eqref{eq:nv-mld}. From the above proof, we see that the equality in \eqref{eq:nv-mld} holds if and only if $\mld(\tx, \tX) = d\cdot \mld(x,X)$ and $\hvol(\tx,\tX) = 9\cdot \mld(\tx, \tX)$. From the above arguments, the latter equality holds if and only if $\tx\in \tX$ is either smooth or a cyclic quotient singularity of type $\frac{1}{3}(1,1,1)$. In particular, $x\in X$ is a quotient singularity. Denote by $\pi: (x'\in X')\to (x\in X)$ the finite quotient morphism where $x'\in X'$ is smooth. Let $G$ be the Galois group of $X'/X$. Then we have $\hvol(x,X) = \frac{27}{|G|}$ which implies $\mld(x,X) = \frac{3}{|G|}$. Suppose $\mld(x,X)$ is computed by a prime divisor $E$ over $x$. Let $\mu: Y\to X$ be a proper birational morphism where $Y$ is normal and contains $E$ as a prime divisor. Let $Y'$ be the normalization of the base change $Y\times_X X'$ and $\mu': Y'\to X'$ the induced proper birational morphism. Then we have a $G$-action on $Y'$ such that the morphism $\pi_Y: Y'\to Y$ is finite and $G$-equivariant. Let $E'$ be an irreducible component of $\pi_Y^* E$. By \cite[Proof of Proposition 5.20]{KM98}, we have $A_{X'}(E') = r A_X(E)$ where $r\leq |G|$ is the ramification index of $\pi_Y$ along $E'$. Thus we have
\[
3 = \mld(x', X') \leq A_{X'}(E') = r A_X(E) = r\cdot  \mld(x,X) = r\cdot \tfrac{3}{|G|}\leq 3,
\]
which implies that $E'$ computes $\mld(x', X')$ and $r = |G|$. It is well-known that the only divisor computing the $\mld$ of a smooth point is the exceptional divisor of the usual blow-up. Thus by changing the birational models $Y$ and $Y'$ we can assume that $Y' = \Bl_{x'} X'$ and $Y = Y'/G$, where $E$ and $E'$ are the unique exceptional divisors of $\mu$ and $\mu'$ respectively. Then $r = |G|$ implies that the $G$-action on $E'\cong \bP^2$ is trivial. In particular, the $G$-action on $x'\in X'$ in suitable local analytic coordinates $(x_1,x_2,x_3)$ only contains multiplication by scalars. Thus $x\in X$ is a cyclic quotient singularity of type $\frac{1}{r}(1,1,1)$. Conversely, if $x\in X$ is of type $\frac{1}{r}(1,1,1)$, then $x\in X$ is locally analytically isomorphic to the affine cone over $\bP^2$ with polarization $\cO(r)$. Hence by taking the blow-up at the cone point we get a log resolution of $x\in X$ with a unique exceptional divisor $E$ such that $\mld(x,X) = A_X(E) = \frac{3}{r} = \frac{\hvol(x,X)}{9}$. Thus the proof is finished.  
\end{proof}

\begin{expl}\label{expl:ADE}
\begin{enumerate}
    \item If $x\in X$ is an $A_k$-singularity, that is, locally analytically given by $x_1 x_2 + x_3^2 + x_4^{k+1} = 0$, then its local volume is $16$ when $k=1$, $\frac{125}{9}$ when $k=2$, and $\frac{27}{2}$ when $k\geq 3$. See \cite[Section 5]{Li18}, \cite[Example 4.7]{LL19} and \cite[Example 7.1.2]{LX20}.
\item If $x\in X$ is a $cA_1$-singularity but not an $A_k$-singularity, then $x\in X$ has to be a transversal $A_1$-singularity with $\hvol(x,X) = \frac{27}{2}$.
\item If $x\in X$ is defined by $x_1 x_2 + x_3^3 + x_4^k=0$ for $3\leq k\leq 6$, then by \cite{CS15} and \cite{LST25} we have $\hvol(x,X) = \frac{4(3+k)^3}{9k^2}$. Thus we get values $\frac{32}{3}, \frac{343}{36}, \frac{2048}{225}, 9$ for $k = 3,4,5,6$. Note that  $k=3,4,5$ correspond to $D_4$, $E_6$, and $E_8$-singularities.
\item If $x\in X$ is a $D_k$-singularity with $k\geq 5$, then we have local equation $x_1x_2 + x_3^2 x_4 + x_4^{k-1} = 0$. Then by \cite[Appendix, Example 4.1]{LX18} we know that its local volume is computed by the monomial valuation $v_{\bfw}|_{K(X)}$ with weight $\bfw=(1, 1, \sqrt{3}-1, 4-2\sqrt{3})$ whose K-semistable Fano cone degeneration is precisely the $D_\infty$-singularity $x_1x_2 + x_3^2 x_4 = 0$. Thus we have $\hvol(x, X) = 6\sqrt{3}$ which is the same as the local volume of a $D_\infty$-singularity.
\item If $x\in X$ is an $E_7$-singularity, then we have local equation $x_1x_2 + x_3^3 + x_3 x_4^3 =0$. By \cite[Section 8]{Li18} and \cite[Example 7.1.3]{LX20} we know that $\hvol(x, X) = \frac{250}{27}$.
\end{enumerate}

\end{expl}

We propose the following questions for future study.

\begin{que}\label{que:vol-list}
Let $\hVol_3$ denote the set of  local volumes of all klt threefold  singularities $x\in X$. By \cite{XZ24} (see also \cite{Zhu24, LMS23}) we know that $0$ is the only accumulation point of $\hVol_3$. Moreover, Theorem \ref{thm:nv-3fold} implies that if $\hvol(x,X)\in \hVol_3 \cap [9, 27]$ then $x\in X$ is either a quotient singularity of order at most $3$ or a hypersurface singularity of type $cA_{\leq 2}$.
 From Example \ref{expl:ADE} we have the following containment
\[
    \hVol_3 \cap [9, 27] \supset \left\{9, \tfrac{2048}{225}, \tfrac{250}{27}, \tfrac{343}{36}, 6\sqrt{3}, \tfrac{32}{3}, \tfrac{27}{2}, \tfrac{125}{9}, 16, 27\right\}.
\]
Is this containment  an equality?
\end{que}

\begin{que}
Let $x\in X$ be the anti-canonical affine cone over $\bF_1$ (which is Gorenstein canonical). Then by \cite[Section 8.3]{Blu18} we have $\hvol(x,X) = \frac{46+13\sqrt{13}}{12}\approx 7.74$. 
Is this the largest local volume among all klt threefold singularities that are not LCIQ, i.e.\ (finite, quasi-\'etale) quotients of local complete intersection singularities?
\end{que}

\begin{que}
Does a similar inequality to \eqref{eq:nv-mld} hold for klt threefold pairs? 
\end{que}

\bibliography{ref}

\end{document}